\documentclass[11pt,a4paper]{amsart}
\pdfoutput=1
\usepackage{color} 
\usepackage{enumerate}
\usepackage[utf8]{inputenc} 
\usepackage[T1]{fontenc}
\usepackage[english]{babel} 
\usepackage{textcomp} 
\usepackage{gensymb} 
\usepackage[font=small,labelfont=bf]{caption} 
\usepackage{mathtools} 
\usepackage{amssymb}
\usepackage{amsthm}
\usepackage{bm} 
\usepackage{lmodern} 
\usepackage{xcolor}
\usepackage{graphicx} 
\graphicspath{{Figures/}}
\usepackage[hidelinks]{hyperref}
\usepackage{verbatim}
\usepackage{mdframed}
\usepackage{tikz}
\usepackage{algorithm2e}
\RestyleAlgo{ruled}

\usepackage{lineno}

\usepackage{float}
\floatstyle{plaintop}
\restylefloat{table}

\usepackage{booktabs}

\usepackage[tmargin=1in,bmargin=1in,lmargin=1in,
    rmargin=1in]{geometry} 
 
\def\beq{\begin{eqnarray}}
\def\eeq{\end{eqnarray}}
\newcommand{\nn}{\nonumber}

\newtheorem{theorem}{Theorem}[section]
\newtheorem{proposition}[theorem]{Proposition}
\theoremstyle{definition}

\newtheorem{lemma}[theorem]{Lemma}
\newtheorem{corollary}[theorem]{Corollary}
\newtheorem{uptheo}{Theorem}[section]
\newtheorem{remark}{Remark}

\newcommand{\R}{\ensuremath{\mathbb{R}}}
\newcommand{\cL}{\ensuremath{\mathcal{L}}}

\newcommand{\bmy}{\ensuremath{\bm{y}}}
\newcommand{\bv}{\ensuremath{\bm{v}}}
\newcommand{\bw}{\ensuremath{\bm{w}}}
\newcommand{\ba}{\ensuremath{\bm{a}}}
\newcommand{\bb}{\ensuremath{\bm{b}}}
\newcommand{\bc}{\ensuremath{\bm{c}}}

\definecolor{blue}{rgb}{0,0,1}
\definecolor{red}{rgb}{1,0,0}
\definecolor{green}{rgb}{0,.6,.2}
\definecolor{purple}{rgb}{1,0,1}
\definecolor{brown}{rgb}{.59,.29,0}
\definecolor{pink}{rgb}{1,0,1}

\DeclareMathOperator{\mca}{MCA}

\hypersetup{
     colorlinks   = true,
     citecolor    = black,
     linkcolor    = blue!50!black,
     urlcolor     = blue!50!black,
}

\begin{document}
\title{Team game adaptive dynamics}

\author[C.J.\@ Karlsson]{Carl-Joar Karlsson}
\address{Department of Mathematical Sciences, Chalmers University of Technology and the University of Gothenburg, 412 96 Gothenburg, Sweden}
\email{carljoar@chalmers.se}

\author[P.\@ Gerlee]{Philip Gerlee}
\address{Department of Mathematical Sciences, Chalmers University of Technology and the University of Gothenburg, 412 96 Gothenburg, Sweden}
\email{gerlee@chalmers.se}

\author[J.\@ Rowlett]{Julie Rowlett}
\address{Department of Mathematical Sciences, Chalmers University of Technology and the University of Gothenburg, 412 96 Gothenburg, Sweden}
\email{julie.rowlett@chalmers.se}

\maketitle

\begin{abstract}
Adaptive dynamics describes a deterministic approximation of the evolution of scalar- and function-valued traits. Applying it to the team game developed by Menden-Deuer and Rowlett~\cite{menden2019theory}, we constructed an evolutionary process in the game. We also refined the adaptive dynamics framework itself to a new level of mathamatical rigor.  In our analysis, we demonstrated the existence of solutions to the adaptive dynamics for the team game and determined their regularity.  Moreover, we identified all stationary solutions and proved that these are precisely the Nash equilibria of the team game.  Numerical examples are provided to highlight the main characteristics of the dynamics. The linearity of the team game results in unstable dynamics; non-stationary solutions oscillate and perturbations of the stationary solutions do not shrink. Instead, a linear type of branching may occur. We finally discuss how to experimentally validate these results.  Due to the abstract nature of the team game, our results could be applied to derive implications and predictions in several fields including biology, sports, and finance.  
\end{abstract}

\thispagestyle{empty}





\section{Introduction}

Biological diversity is an essential part of nature --- not only by its own value but also because it provides stability to ecosystems and pleasurable environments \cite{abe2012biodiversity}.
It can be enhanced via a number of mechanisms relying on external factors, such as the amount of green area in cities, but it is also a consequence of the evolutionary process itself \cite{abe2012biodiversity}. 
There is however a need to better understand the intrinsic processes that lead to biological diversity in the absence of external influence. 
Mathematical modeling and prediction has often led to unrealistic conclusions, for example that the number of 
species cannot exceed the number of limiting resources \cite{armstrong1980competitive}. 
Similarly, the so called exclusion principle says that two competing species that occupy the same ecological niche cannot co-exist \cite{hardin1960}. 
For marine microbes, these predictions are terribly wrong \cite{menden2014many}. 
In fact, the number of microbial species by far exceeds the predictions from competition theory, and there is tremendous variability between and within species \cite{olofsson2019high,menden2015bloom,menden2014many,hutchins2013taxon,thomas2012global,ward2002many}.
This discrepancy between theory and reality is known as the ``paradox of the plankton'' in marine ecology and it has gained considerable attention \cite{hutchinson1961paradox}. 
Attempts to resolve the paradox abound. 
For instance, Huisman and Weissing~\cite{huisman1999} demonstrated that common competition models can sustain a system with large numbers of species by oscillating, or cycling, the population sizes.
This would partly solve the paradox of the plankton, since the oscillating populations can coexist in much higher numbers than predicted by steady-state analysis.  
A different approach was taken by Menden-Deuer and Rowlett~\cite{menden2019theory} (see also \cite{menden2014many,menden2021biodiversity}) when they modeled the inter-species competitions among cloning (i.e., asexual) microbes using non-cooperative game theory. 
In the game that they developed, an unlimited number of species may coexist. 

These resolutions to the paradox of the plankton offer a phenomenological explanation to why an ecosystem can support a large number of species.
However, even though non-steady state analysis and game theory seem to solve the paradox of the plankton, it does not explain how the species evolved to the current state in the first place or what happens if coexisting species would evolve further.
In other words, the current knowledge can answer the question of why a set of species can co-exist but there is a need to understand how such an ecosystem evolves as the species are evolving.

One theory that has provided insight to the question of how some species evolve into the observable, extant ones is evolutionary game theory (EGT). 
In EGT, an organism's actions and behaviors are represented by a \textit{strategy}.
EGT has explained a wide range of animal behaviors by modeling fitness as a function of strategies that can be observed in populations.
For instance, Maynard Smith and Price~\cite{smith1973logic} explained why some animals do not harm each other in fights against members of their own species by showing that this behavioural strategy corresponds to stable maxima of the fitness. 
At a stable maximum of the fitness, the strategies are known as evolutionarily stable strategies (ESS) since small changes to the strategies are not beneficial. 
However, it remained to be explained how such strategies can appear as a result of evolution.
In the twenty-five years following Maynard Smith's and Price's study of ESS, it was observed that although ESS are long-term stable, they might not evolve spontaneously as a result of small changes to the strategies of the population \cite{eshel1981kin,eshel1983evolutionary,taylor1989evolutionary,christiansen1991conditions,abrams1993}.
This enigma has resulted in the development of a mathematical framework called \emph{adaptive dynamics}, which assumes that small changes to strategies can make permanent change to the population's choice of strategy whenever a mutant carrying the new strategy has a positive invasion fitness. Such permanent change provides a mathematical representation of natural selection. 

Geritz \textit{et al.}~\cite{geritz1996evolutionarily}  classified eight scenarios within adaptive dynamics when the population is close to an ESS. 
For instance, they stated sufficient conditions for the ESS to be an attractor in the sense that strategies that are very similar to the ESS converge towards the ESS. 
On the contrary, if the ESS is not an attractor and if there are multiple successful strategies that are similar to the ESS, then it could happen that strategies close to the ESS are ``branching'' into multiple strategies. 
Branching is an important mechanism for diversification within ecosystems, and a possible route to speciation, and we will address this in Section~\ref{sec:discussion}. 

The classification of Geritz \textit{et al.}~\cite{geritz1996evolutionarily} is limited to a certain type of strategies, namely real, scalar-valued strategies or vector-valued strategies.
These are not the only possibilities, and for our purposes, it is important to investigate strategies beyond scalar-valued or vector-valued ones.
Strategies can be chosen in a variety of ways, and it is a key challenge for the researcher to construct a suitable class of strategies. 
The simplest strategies are percentages, mass, time and other quantities that can be represented by a scalar value.
In ecology, strategies usually represents traits of organisms. 
In cases when adaptive traits are best described by a variation along a continuum, such as the distribution of age or weight within a population, the traits can be described mathematically by \textit{functions}. 
Evolving functions is mathematically challenging, but on the other hand, function-valued traits are applicable to many contexts \cite{dieckmann2006adaptive}. 
To our knowledge, there is no classification similar to the one by Geritz \textit{et al.}~\cite{geritz1996evolutionarily} of function-valued adaptive dynamics.
However, some studies of function-valued traits have characterized ``uninvadable'' species by other approaches, such as optimal control theory \cite{parvinen2013,avila2021} and Lagrangian dynamics \cite{ito2016evolutionary} or more general variational principles \cite{metz2016,kuzenkov2015,morozov2016}.
In some studies, function-valued traits are modelled using a trauncated basis of functions, which are selected for computational reasons \cite{gao2019}. We propose a type of gradient flow method, which can be applied to infinite dimensional vector spaces and subsets of those, if treated carefully. The game dynamics studied here is interesting from a mathematical point of view for several reasons.  First, it is infinite-dimensional and therefore requires sophisticated treatment in order to give reasonable results. Furthermore, it may help to illuminate the aforementioned classification problem.

We therefore propose that adaptive dynamics is a suitable framework for studying the evolution of the team game introduced in  \cite{menden2014many,menden2019theory}.
This game was initially constructed for vector-valued strategies, which we may refer to as the discrete team game.  It was generalized to function-valued strategies in subsequent publications \cite{menden2021biodiversity, rowlett2022diversity}, hence the need to utilize the theory of adaptive dynamics of functions. 
In both cases---vector-valued as well as function-valued---the strategy represents a \textit{composition} of the species.
Due to its applicability in a broad range of context where the composition of a unit of members is studied, the game is called ``the Game of Teams,'' or the team game, in this article as well as in \cite{rowlett2022diversity}.
It will be introduced in detail in Section~\ref{sec:background} along with a useful collection of notions and ideas from adaptive dynamics. 
In Section~\ref{sec:analysis}, we analyze the behavior of the dynamics theoretically and purely mathematically for the function-valued team game.
Then, once the mathematical results are in place, we focus on applying adaptive dynamics to the function-valued game in Section~\ref{sec:game-dynamics}.
Section~\ref{sec:computations} 
analyzes the adaptive dynamics for the discrete, vector-valued team game.  We obtain the explicit form of the solutions to the dynamical system and also give examples and compare the results to the function-valued game in Section~\ref{sec:game-dynamics}.
Finally, we present experimental conditions that could test the current results.

\section{Background \label{sec:background} }
Game theory has advanced our understanding of decision making, animal behaviour, population dynamics and other phenomena involving actions performed by humans or animals \cite{mazalov2014}. 
A central notion in non-cooperative game theory is the \textit{payoff} to a player as a function of the actions it takes together with the actions taken by other players.
For example, in the two-player rock-paper-scissors, the pair (rock, scissors) would give a win to the first player and a loss to the second player.  To express this mathematically, the payoff function could give $+1$ to the first player and $-1$ to the second player.

A \textit{mixed strategy} is a probability distribution over the set of actions.  In our example of rock-paper-scissors, each player could instead choose the probability of drawing either rock, paper, or scissors.  If they draw these at random, this would be represented by the mixed strategy $(1/3, 1/3, 1/3)$.
There are three \em pure strategies \em in this example:  $(1, 0, 0)$, $(0,1,0)$, and $(0,0,1)$ corresponding to always drawing rock, paper, or scissors, respectively.  The payoffs to each player are in this cased calculated using the payoffs for all combinations of pure strategies together with the definition of expected value according to the probabilities with which the players choose to execute the pure strategies.
In this article, the strategies can always be seen as mixed strategies because there is randomness in the game. In order to clarify this, we need to consider the game in more detail.

\subsection{The discrete game of teams}  \label{sec:got}
The team game of Menden-Deuer and Rowlett~\cite{menden2019theory} was initially developed to investigate species of asexually reproducing microbes competing for survival.  It was later generalized and interpreted in other contexts  \cite{menden2021biodiversity,rowlett2022diversity}. 
We therefore may use the terms species and teams interchangeably. To describe this game, we consider a collection of species, each consisting of several individuals. 
Each individual has a ``strength'' which can be measured and compared with another individual's strength.  This strength is known as a \emph{competitive ability} \cite{menden2019theory}, abbreviated CA, and is selected from the values $\{k/M\}$ for the integers $k=0, \ldots, M$.  For simplicity, assume that there are just two species. The species compete in such a way that one randomly chosen individual from one species competes against an individual from the other team, which is also chosen at random. The stronger individual defeats the weaker so that it can replicate while the losing individual dies.  Nothing happens if the competitors are equally strong.  
Thus, individual success implies population growth of the species to which the winning individual belongs while the losing species experiences a population decrease. 
This individual competition repeats, and the cumulative losses and gains can result in either one's extinction and the other's dominance or co-existence.  

Let $y_k$ be the number of individuals in species $\bmy$ with competitive ability equal to $\frac k M$, and $z_k$ be the number of individuals in species $\bm{z}$ with competitive ability equal to $\frac k M$.  For this to be meaningful we assume 
\beq 
y_k, z_k \in [0, \infty), \quad \forall 0 \leq k \leq M, \quad \sum_{k=0} ^M y_k > 0, \quad \sum_{k=0} ^M z_k > 0. \label{eq:discrete_setup} 
\eeq 
Note that $y_k$ and $z_k$ need not be integer-valued.  Then, the payoff in the team game as described above to $\bmy$ in competition with $\bm{z}$ is 
\begin{equation}
E[\bmy, \bm{z}] = \sum_{k=0}^M y_k \left(\sum_{j=0}^{k-1} z_j - \sum_{\ell=k+1}^M z_\ell\right).
	\label{eq:expect-discrete}
\end{equation}
The payoff to $\bm{z}$ in competition with $\bmy$ is computed analogously, by summing over the cumulative wins and losses, so that 
\[ E[\bm{z},\bmy] = \sum_{k=0}^M z_k \left(\sum_{j=0}^{k-1} y_j - \sum_{\ell=k+1}^M y_\ell\right). \] 
It is straightforward to compute that this game is zero-sum and symmetric.  
Each competition between two teams has randomly selected individuals competing, but $E[\bmy,\bm{z}]$ captures the statistical success of the competing teams and can be analyzed without addressing the randomness of the game.  Identifying each CA value as a pure strategy, the vector $\bmy=(y_0, \ldots, y_M)$, suitably normalized, can be identified with a mixed strategy, and the game can be expressed in normal form such that $E[\bmy,\bm{z}]$ is the payoff to $\bmy$ in competition with $\bm{z}$ computed according to the definition of expected value. 

For the game to be fair and interesting, we impose that every team needs to respect a bound on its mean strength, or \em mean competitive ability, \em abbreviated MCA.
This MCA for species $\bmy$, as well as its constraint are respectively 
\begin{equation}
\mca(\bmy) := \frac{\sum_{k=0}^M \frac k M y_k}{\sum_{k=0}^M y_k} \leq \frac 1 2.
\label{eq:discrete-mca}
\end{equation}
The same constraint is imposed on any other species that competes.  We will identify a team with its strategy, since the strategy of the team fully characterizes and distinguishes the team.  The strategy can in term be uniquely identified with a vector in $\R^{M+1}$ whose components satisfy \eqref{eq:discrete_setup} and \eqref{eq:discrete-mca}.  We refer to these as the \textit{discrete} strategies and the corresponding game as the \em discrete team game.  \em  

One way to create a team is to compute the sum of two teams, meaning that we compute the sum of their strategies, because the resulting strategy will satisfy both \eqref{eq:discrete_setup} and \eqref{eq:discrete-mca}.  In this way, one may also consider any number of competing species by letting each species compete against the sum of all the others.  The only constraints on the composition of the teams are \eqref{eq:discrete_setup} and \eqref{eq:discrete-mca}; they are otherwise allowed to be chosen freely.

To identify those strategies that may be more likely to win in competition with others (or less likely to lose), we recall an important notion in game theory, an \em equilibrium point, \em also known as a \em Nash equilibrium point \em due to Nash's proof of their existence \cite{nash1950}. An equilibrium point is a collection of strategies for all competing teams so that if any one team alone changes their strategy, their payoff does not increase. Menden-Deuer \textit{et al.} \cite{menden2021biodiversity} identified all equilibrium points for the discrete game of teams.  We summarize the result here.  

\begin{theorem}[See Theorem 1 in \cite{menden2021biodiversity}] \label{th:discrete_eqs} In the discrete game of team as defined here, assume first that $M$ is odd.  Then an equilibrium point consists of strategies that are a positive scalar multiple of the vector $(1, 1, \ldots, 1)$.  If we instead assume that $M$ is even, then an equilibrium point consists of strategies that are of the form $(a,b,a,b, \ldots, a)$ for two non-negative constants $a$ and $b$ that are not both zero.  
\end{theorem}

The phenomenon that the shape of equilibrium strategies depends on the discretization of the game motivates one to consider a game in which the competitive ability values can be selected from the entire (continuous) range of values $[0,1]$. 

\subsection{The function-valued game of teams}  \label{sec:got_bm_c}
Completely analogous to the discrete strategies are the \textit{continuous} and \textit{bounded measurable} strategies, introduced by Menden-Deuer \textit{et al.}~\cite{menden2021biodiversity}. For a continuous (respectively bounded measurable) non-negative function defined on $[0,1]$, we use the measure $f(x)dx$ with $dx$ the one-dimensional Lebesgue measure to define the amount of individuals of the associated team having competitive ability within any given subinterval of $[0,1]$.  Analogous to the discrete game, we identify a team with its strategy, that is a function satisfying 
\beq f:[0,1] \to [0, \infty), \quad \int_0 ^1 f(x) dx > 0, \quad \mca(f) = \frac{\int_0 ^1 x f(x) dx}{\int_0 ^1 f(x) dx} \leq \frac 1 2. \label{eq:bm_cont_setup} \eeq 
In the continuous game, we assume further that the function is continuous, whereas in the bounded measurable game, we only assume further that the function is in $L^\infty[0,1]$.  We refer to both of these games as \em function-valued games of teams. \em

Figure~\ref{fig:concept} visualizes the amount of individuals with $a<\mathrm{CA}<b$ in the gray area.
The payoff to a strategy $f$ in competition with a strategy $g$ is in this case
\begin{equation*}
 E[f,g] = \int_0^1 f(x) \left( \int_0^x g(y)\,dy-\int_x^1 g(y)\,dy \right) dx
\end{equation*}

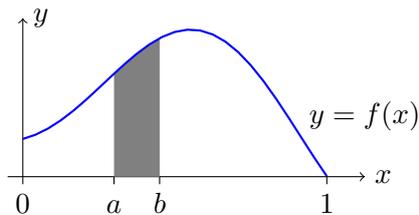
\begin{figure}
\centering
\begin{tikzpicture}[domain=0:4]
  \draw[->] (-0.2,0) -- (4.5,0) node[right] {$x$};
  \draw[->] (0,-0.1) node[below] {$0$} -- (0,2.1) node[right] {$y$};
  \draw     (1.2,0)  -- (1.2,-0.1);
  \draw     (1.8,0)  -- (1.8,-0.1);
  \node at (1.2,-0.615) [above] {$a$};
  \node at (1.8,-0.615) [above] {$b$};
  \draw     (4.0,0)  -- (4,-0.1) node[below] {$1$};
  \fill [gray, domain=1.2:1.8, variable=\x]
      (1.2, 0)
      -- plot ({\x}, {0.5*(1 +0.51*\x + \x *sin(\x r))})
      -- (1.8, 0)
      -- cycle;
  \draw[color=blue, thick] plot (\x, {0.5*(1 +0.51*\x + \x *sin(\x r))});
  \node at (4.5,0.8) {$y=f(x)$};
\end{tikzpicture}
\caption{The strategies of the team game are distributions of competitive ability. For the function-valued game, in the interval $(a,b)$ is the number (or percentage) of individuals with competitive ability between $a$ and $b$.}
\label{fig:concept}
\end{figure}

The team game in this case also generalizes to multiple teams analogously to the discrete game of teams.  Specifically, $f$ competes against $n$ other teams $g_1,...,g_n$ by simply competing with the strategy defined by the sum of the other teams, noting that such a strategy satisfies \eqref{eq:bm_cont_setup}.  Menden-Deuer \textit{et al.}~\cite{menden2021biodiversity} identified the equilibrium points for these games of teams as well.  

\begin{theorem}[See Theorem 1 in \cite{menden2021biodiversity}] \label{th:eqstr_function}
In both the continuous and bounded measurable games of teams defined here, all equilibrium points are collections of strategies for all teams that are positive constant functions, which are not necessarily identical.  
\end{theorem} 

The restriction to $[0,1]$ in all games of teams defined here can be relaxed; the functions can be supported on any compact subset of the real line.  However, the unit interval is convenient, and if a function is defined on any other compact interval on the real line, then it can be transformed via a change of variables to a function on $[0,1]$.  So, no generality is lost by making this assumption.   For details, see \cite{rowlett2022diversity}.

\subsection{Adaptive dynamics} \label{sec:ad-intro}
At the heart of adaptive dynamics lies the assumption that there exists a ``resident'' population in which there can appear mutations and that the success of the mutants can be inferred from the initial growth rate of the mutated individuals. 
The initial growth rate of mutants is called \textit{invasion fitness}. 
It is further assumed that mutations are rare, such that each mutant either takes over the entire population or goes extinct before the next mutant arrives \cite{metz1996adaptive,geritz1996evolutionarily,brannstrom2013hitchhiker}.
In other words, if a mutant has lower fitness than the resident population then it disappears, but if the mutant's fitness is higher than the resident population's then it is assumed that the mutation spreads into the entire resident population.
By assuming this, adaptive dynamics offers a deterministic description of biological evolution. 

Dieckmann \textit{et al.}~\cite{dieckmann2006adaptive} proposed a framework for adaptive dynamics on function-valued traits. 
They used approximations to stochastic models, assuming that (a) mutations make small changes to the traits, and that (b) the natural selection occurs much faster than the typical time between the appearance of novel mutations, so that each population is monomorphic.
The result of these considerations is that a trait function $f$ develops according to 
\begin{equation}
	\frac{d}{dt} f(x) = \frac{1}{2}\mu_f \bar{n}_f \int_\Omega \sigma^2_f(x,y) g_f(y)\,dy
\label{eq:master}
\end{equation}
where $f$ is the trait/strategy.
The integration domain $\Omega$ needs to be selected to suit the model. The quantity $\mu_f$ is the probability that $f$ can be reached via mutations of nearby strategies, and $\bar{n}_f$ is the equilibrium population size, which is assumed to be constant and independent of the strategy $f$.
Here, $\sigma^2_f$ is the variance-covariance function of the mutation distribution. The role of the variance-covariance function $\sigma^2_f$ is to account for cross-dependence; if the dynamics at $x$ changes the strategy $f$ in such a way that it affects $f$ at another point $y$, this is encoded in the variance-covariance function.
Typically this is formulated as a constraint on all traits.
The function $g_f$ is the functional gradient of the invasion fitness function.
Let $S$ be the set of strategies and let $E(f,g)$ be the invasion fitness of $f\in S$ in the resident population with trait $g\in S$. 
Then
\begin{equation*}
g_f(x) = \left.\frac{d}{dt}\right|_{t=0} \, E(f+t\delta_x,f).
\end{equation*}
Here, $\delta_x$ is the Dirac delta distribution \cite{friedlander1998introduction} centered at $x$.
We will call the dynamics equation~(\ref{eq:master}) the \emph{canonical equation} of adaptive dynamics of function-valued traits.

\subsection{Introducing the function-valued team game in the context of adaptive dynamics} \label{sec:combine}

Consider a collection of several species: $f_1, \ldots, f_n$, each of which is characterized by its strategy $f_i$.  A situation when they all compete in the same game can be interpreted as a competition between species (or strains of a species) for resources or in combat or similar situations. 

Thus, it is in complete analogy with Geritz \textit{et al.}~\cite{geritz1996evolutionarily} and Dieckmann \textit{et al.}~\cite{dieckmann2006adaptive} that we let $E[f_i,f_1+f_2+...+f_n]$ be the growth rate of a species (or strain) $i$ in competition with all the other species.   
Notice that 
\begin{multline}
E[f_i,f_1+f_2+...+f_n] = E[f_i,f_1]+E[f_i,f_2]+...+E[f_i,f_n]\\
= E[f_i,f_1]+E[f_i,f_2]+...+E[f_i,f_{i-1}]+E[f_i,f_{i+1}]...+E[f_i,f_n]
\end{multline}
since $E[f_i,f_i]=0$. 
Moreover, if $a$ is a constant, and $f,g$ are two integrable functions then $E[a f,g]=a E[f,g]$ and $E[f,a g]= a E[f,g]$. 
The selection gradient is therefore given by 
\begin{equation}
\label{eq:gf}
g_f(x)=\left.\frac{d}{dt}\right|_{t=0} E[f+t\delta_x,f] = \int_0^x f(y)\,dy - \int_x^1 f(y)\,dy.
\end{equation}
The canonical equation of adaptive dynamics, equation~(\ref{eq:master}), now reads
\beq
\frac{d}{dt} f(x) = \frac{1}{2}\bar{n}_f\mu_f\int_0^1 \sigma^2_f(x,y)\left(\int_0^y f(z)\,dz - \int_y^1 f(z)\,dz\right)dy. \label{eq:barn}
\eeq
We will assume that $\bar{n}_f$ is a constant, since it appears as a prefactor in equation~(\ref{eq:master}) and therefore only impacts the \emph{rate} of change and not the \emph{direction}.
We assume that the competitive ability does not affect the mutation rate, and hence the mutation probability $\mu_f$ in the canonical equation~(\ref{eq:master}) will also be assumed to be constant. As there are no physical time units in the canonical equation, we may set $\frac{1}{2}\bar{n}_f\mu_f=1$ without losing any information.  Notice that in the canonical equation \eqref{eq:master}, we have $\Omega=[0,1]$, which is the logical choice for the team game as explained in Section~\ref{sec:got}.

In this case, and if it is further assumed that there is no variance or covariance, the function-valued strategies would develop in time according to
\[
\frac{d}{dt}f = g_f(x).
\]
This clearly does not take into account that $f$ may be subjected to model-specific constraints. In \S \ref{sec:constraints}, we resolve that problem. In computing the selection gradient, see equation~(\ref{eq:gf}), we use the Dirac delta, which is a distribution.
However, it is also possible to derive the expression for the selection gradient working with the function spaces $L^p[0,1]$ for $1 \leq p \leq \infty$ since these spaces contain the strategies in the function-valued team game.  Since continuous and bounded measurable functions on the compact interval $[0,1]$ are contained in 
$L^2[0,1]$, we may use the $L^2$ inner product on $L^2$ functions, denoted by $\langle\,,\rangle$, to compute the selection gradient.
Computing the selection gradient amounts to taking the functional derivative of $E[f,g]$ at $g=f$. 
That is, $\langle\nabla E(f),v\rangle = \left.\frac{d}{dt}\right|_{t=0} E[f,f+tv]$, for any $v$ in $L^2[0,1]$.
By computing this for arbitrary $v$ we find the selection gradient 
\begin{equation}
\nabla E(f)(x) = \int_0^x f(y)\,dy - \int_x^1 f(y)\,dy.
\label{eq:selection-grad}
\end{equation} 
The selection gradient $\nabla E$ maps a function $f:[0,1]\to[0,\infty)$ onto the difference between the integral of $f$ over $[0,x]$ and the integral of $f$ over $[x,1]$.
We let~(\ref{eq:selection-grad}) define $\nabla E(f)$ for any strategy $f$ of a continuous variable, while the discrete strategies' selection gradient is given in Section~\ref{sec:computations}.

The definitions of continuous and bounded measurable strategies \eqref{eq:bm_cont_setup} requires that the function $f$ be non-negative. Yet, in the dynamics there has to be a possibility that the function $f$ decreases at some $x\in[0,1]$.  Consequently, the dynamics cannot be restricted to the space of strategies. 
If $\nabla E(f)(x)\geq 0$ would be true for all $x\in[0,1]$ then the only possible change to $f$ would be that it grows. 
Therefore, we will consider the adaptive dynamics for functions in the $L^p[0,1]$ spaces for $1\leq p\leq \infty$. Although this approach does not preserve non-negativity, we can make the dynamics respect the MCA constraint \eqref{eq:bm_cont_setup}.  Similarly, we should necessarily ensure that $f$ stays measurable and bounded under the dynamics if $f$ was measurable and bounded to begin with.  These requirements are treated in the following.

\subsubsection{Global inequality constraints}\label{sec:global}
The procedures that are needed to deal with inequality constraints are described by Dieckmann \textit{et al.}~\cite{dieckmann2006adaptive}. 
Since only inequality constraints are treated in the context of the current work, we focus on such constraints here.  The global inequality constraints are of the form 
\[ w(f)\leq 0 \]
for all $f$ under consideration (eg $f \in L^p[0,1]$ for some $1 \leq p \leq \infty$).  Here $w$ maps the function space under consideration to $\R$ and is chosen based on the physical constraints in the modeling situation. 
Starting from a variance-covariance function $U_f(x,y)$, the following transformed variance-covariance function ensures that $w(f)\leq 0$ is satsified by the dynamics:
\[
\sigma^2_f(x,y)=\int_\Omega\int_\Omega \tilde{P}(x,r)U_f(r,s)\tilde{P}(s,y)\,dr\,ds.
\]
Here, the projection $\tilde{P}$ is defined by the equation 
\[
\tilde{P}(x,y)=\delta_x(y)-\tilde{N}_f(x)\tilde{N}_f(y)H(w(f))H\left(\int_\Omega g_f(z)\tilde{N}_f(z)dz\right),
\]
where $H$ is the Heaviside function (i.e., the indicator function supported on $x\geq 0$) and
\[
\tilde{N}_f(x)=\frac{N(x)}{\sqrt{\int_\Omega (N(y))^2\,dy}},\quad N(x)=\left.\frac{d}{dt}\right|_{t=0} w(f+t\delta_x).
\]
Global \emph{equality} constraints can be accounted for by removing the factor $H(w(f))$ from the above expression for $\tilde{P}(x,y).$  Notice that Dieckmann \textit{et al.}~\cite{dieckmann2006adaptive} use the opposite sign convention on the inequality constraint, which implies that also $N$ has the opposite sign in this presentation.

\section{Adaptive dynamics for the function-valued team game}  \label{sec:analysis}

In this section, we first consider the constraints on the strategies.  Next, in \S \ref{sec:kernel}, we prove a handful of results concerning the right hand side of the canonical equation~(\ref{eq:master}). This helps us understand the dynamics, which is the topic of \S \ref{sec:existence}.  Finally, we investigate the implications of these results for the evolution of strategies in the function-valued team game in \S \ref{sec:game-dynamics}.

\subsection{The team game adaptive dynamics and constraints} \label{sec:constraints}

As explained above, the selection gradient maps the non-negative elements of $L^\infty[0,1]$, denoted $\cL^\infty_+[0,1]$ into $L^\infty[0,1]$ (i.e., not into $\cL^\infty_+[0,1]$).
Since $L^\infty[0,1]$ is a subspace of $L^2[0,1]$, the inner product of $L^2$ can be used to project the selection gradient onto the subspace of functions that satisfy the MCA constraint. This constraint can be expressed as 
\begin{equation}
w(f)\leq 0\quad\text{ for }\quad w(f)=\int_0^1 (x-1/2)f(x)\,dx.
\label{eq:w}
\end{equation}
The projection onto the tangent of the boundary $w(f)=0$ is thus
\begin{equation*}
P(f) = \frac{\langle f, \nabla\hspace{-0.15em} w\rangle}{\|\nabla\hspace{-0.15em} w\|^2} \nabla\hspace{-0.15em} w,\qquad \nabla w(x)=x-\frac 1 2.
\end{equation*}
Here, $\langle \, ,\rangle$ is the $L^2$ inner product on $[0,1]$, and $\|\nabla w\|^2=\langle \nabla w,\nabla w\rangle$. 
Figure~\ref{fig:levelset} visualizes the level set defined by $w(f)=0$ and the vectors that are parallel and orthogonal to the same level set.
The projection of the selection gradient $\nabla E$ onto the normal direction of $w(f)=0$ is given by
\begin{equation}
P\big(\nabla E(f)\big)(x) = \frac{\int_0^1 (y-1/2)\left(\int_0^y f(z)\,dz-\int_y^1 f(z)\,dz\right)dy}{\int_0^1 (y-1/2)^2 dy}(x-1/2).
\label{eq:pi}
\end{equation}
Notice that $P$ maps any function onto a linear function on $\R$.
Removing the component of the selection gradient $\nabla E(f)$ which is normal to $w(f)=0$ is achieved by projection with $1-P$, where $1$ is the identity mapping. 
This ensures that the constraint in \eqref{eq:w} is respected at all times.

\begin{figure}
  \centering
  \includegraphics[width=0.5\columnwidth]{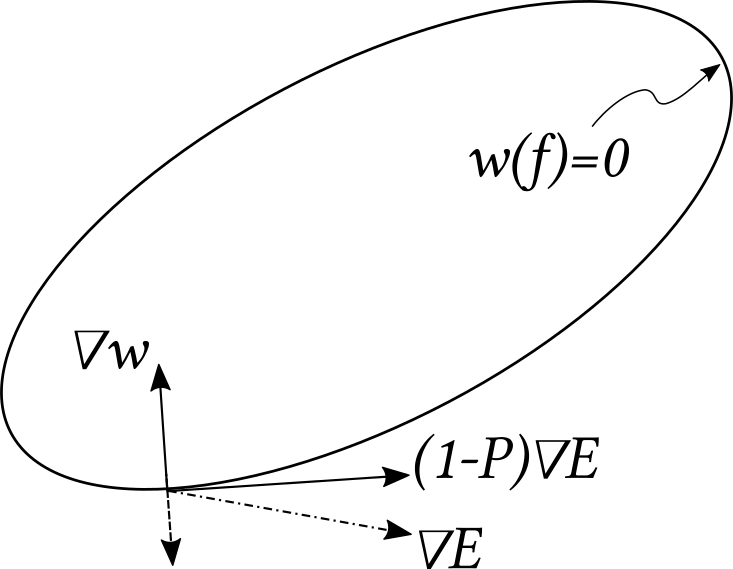}
  \caption{Projection onto the set of functions constrained by $w(f)=0$. The selection gradient belongs to a subspace of $L^2[0,1]$ and the normal component, $P(\nabla E)$, is removed. Thus, $w\big((1-P)\nabla E\big)=0.$}
  \label{fig:levelset}
\end{figure}

These results are consistent with the treatment of global inequality constraints as described by Dieckmann \textit{et al.}~\cite{dieckmann2006adaptive}. 
We support this claim here.
First, notice that 
\[
H\left(\int_\Omega g_f(x)\tilde{N}_f(x)dx\right) = H\left(\sqrt{12}\int_0^1 \left(x-\frac{1}{2}\right)\left(\int_0^x f(y)\,dy-\int_x^1 f(y)\,dy\right)dx\right).
\]
By integration by parts 
\begin{equation}
\label{eq:ibp-grad}
\int_0^1 (x-\tfrac{1}{2})\left(\int_0^x f(y)\,dy-\int_x^1 f(y)\,dy\right)dx = -\int_0^1 (x^2-x)f(x)\,dx\geq 0.
\end{equation}
Therefore, $H(\int_\Omega g_f(x)\tilde{N}_f(x)dx)=1$.
Then
\[
\tilde{P}(x,y)=\delta_x(y)-H(w(f))12(x-\tfrac{1}{2})(y-\tfrac{1}{2}).
\]
Here, $w$ is defined in~(\ref{eq:w}).
This and the general framework in Section~\ref{sec:global} gives
\begin{multline*}
\sigma^2_f(x,y)=\int_0^1\left(U_f(x,s)-12H(w(f))(x-\tfrac{1}{2})\int_0^1(r-\tfrac{1}{2})U_f(r,s)\,dr\right)\tilde{P}_f(s,y)ds\\
=U_f(x,y)-12H(w(f))(x-\tfrac{1}{2})\int_0^1(r-\tfrac{1}{2})U_f(r,y)\,dr\\
-12H(w(f))(y-\tfrac{1}{2})\int_0^1(s-\tfrac{1}{2})U_f(x,s)\,ds\\
+12^2H(w(f))^2(x-\tfrac{1}{2})(y-\tfrac{1}{2})\int_0^1\int_0^1(r-\tfrac{1}{2})U_f(r,s)\,dr (s-\tfrac{1}{2})ds.
\end{multline*}
Notice that $H(w(f))^2=H(w(f))$ by the definition of the Heaviside function.
In the special case $U_f(x,y)=\delta_x(y)$ we obtain
\begin{multline*}
\sigma^2_f(x,y)=\delta_x(y)-12H(w(f))(x-\tfrac{1}{2})(y-\tfrac{1}{2})
-12H(w(f))(x-\tfrac{1}{2})(y-\tfrac{1}{2})\\
+12^2H(w(f))(x-\tfrac{1}{2})(y-\tfrac{1}{2})\underbrace{\int_0^1(s-\tfrac{1}{2})^2ds}_{=1/12}
=\delta_x(y)-12(x-\tfrac{1}{2})(y-\tfrac{1}{2}) H(w(f)).
\end{multline*}
Applying this variance-covariance function is equivalent to the projection by $1-P$, where $P$ is defined in~(\ref{eq:pi}). 
The full adaptive dynamics of the team game constrained to MCA$(f)\leq \frac{1}{2}$ is given by the initial-value problem 
\begin{equation}
\frac{\partial}{\partial t} f = (1-H(w(f))P)\nabla E(f),\qquad\text{ with } \left. f\right|_{t=0} = f_0.
\label{eq:main}
\end{equation}
%
Here $f=f(x,t)$ is a one-parameter family of functions, depending on the parameter $t$, mapping $ x\in [0,1]$ to $\R$, and $P$ is defined by \eqref{eq:pi}.  We refer to $(1-P)\nabla E$ as the \em constrained selection gradient, \em whereas $\nabla E$ is the \em unconstrained selection gradient. \em  With initial conditions, \eqref{eq:main} is a Cauchy problem.  We will show in \S \ref{sec:existence} that for every initial condition in $L^p[0,1]$, for all $1 \leq p \leq \infty$ the solution is a curve $\alpha(t)$ with time-parameter $t$ that remains in the same $L^p$ space.  In the language of ODE theory \cite{lang1993} equation~(\ref{eq:main}) defines an integral curve of the vector field $(1-P)\nabla E$.

\subsection{Results on the constrained and unconstrained selection gradients}  \label{sec:kernel}

In the adaptive dynamics framework, the selection gradient is the driving force behind the evolution of the function-valued traits. 
If the initial data fulfil the criteria to be strategies in the team game, then equation \eqref{eq:main} describes their adaptive dynamics.
The mathematical properties of the selection gradient predict the behavior of the evolution on both long and short timescales.
A function $f$ from $[0,1]$ to $\R$ is mapped by the (constrained or unconstrained) selection gradient onto a function with properties that depend on the original properties of $f$, as shown in Lemma \ref{lem:bound-p}.

We begin by calculating basic properties of the unconstrained selection gradient.  It is a linear: 
\begin{equation*}
\nabla E(f+g)(x) = \nabla E(f)(x) + \nabla E(g)(x).
\end{equation*}
Moreover, $\nabla E(f)$ is a non-decreasing function, which is clear once we rewrite~(\ref{eq:selection-grad}) as 
\beq
\nabla E(f)(x) = 2\int_0^x f(y)\,dy - \int_0^1 f(y)\,dy. \label{eq:sg_nondecreasing}
\eeq

The unconstrained selection gradient $\nabla E$ is an integral operator with kernel $s(x,y)$ defined by 
\begin{equation}
s(x,y) = \chi_{[0,x)}(y)-\chi_{(x,1]}(y) = 
\begin{cases}
-1, & x<y,\\
1, & x>y,
\end{cases} 
\label{eq:E-kernel}
\end{equation}
so that 
\[
 \nabla E(f)(x) = \int_0^1 s(x,y)f(y)\, dy.
\]
The kernel is constant above and below the diagonal $x=y$. 
It is weakly singular on the diagonal, that is, it is undefined on the set $x=y$ with $x,y\in[0,1]$.

The constrained selection gradient may also be defined as a kernel operator.  For this it is convenient to introduce the notation 
\beq A f(x) &=& (1-H(w(f))P) \nabla E f(x)  \nn \\ 
&=& \int_0^x f(y)\,dy -\int_x^1 f(y)\,dy \nn \\
& &-12 H(w(f)) (x-\tfrac{1}{2})\int_0^1\left(y-\frac{1}{2}\right)\left(\int_0^y f(z)\,dz -\int_y^1 f(z)\,dz\right)dy.
\label{eq:def-A}
\eeq 
%
Above, $H$ is the Heaviside function, $12 = \langle x-\frac{1}{2},x-\frac{1}{2}\rangle^{-1}$ is a normalization factor, and $w$ is defined in \eqref{eq:w}.
In order to define $A$ as a kernel operator, recall the integration by parts from equation~(\ref{eq:ibp-grad}), and let $s$ be the kernel of $\nabla E$ as in~(\ref{eq:E-kernel}).  As a kernel operator, $A$ is then given by
\begin{equation}
\begin{split}
&Af(x)=\int_0^1 k(x,y)f(y)\,dy,\\ &\text{with}\quad 
k(x,y) = s(x,y)+12 H(w(f))\left(x-\tfrac{1}{2}\right)\left(y^2-y\right).
\end{split}
\label{eq:A-kernel}
\end{equation}
Notice that if we accept distributions in our theory we may write the mapping $A$ as a kernel operator on the gradient:
\begin{equation}
\begin{split}
&Af(x)=\int_0^1 \sigma^2(x,y)\nabla E(f)(x)\,dy,\\
&\text{with}\quad
\sigma^2(x,y) = \delta_x(y)-12H(w(f))\left(x-\tfrac{1}{2}\right)\left(y-\tfrac{1}{2}\right).
\end{split}
\label{eq:A-of-gradient}
\end{equation}
Here, we denote the kernel by $\sigma^2$, since that correctly describes the connection to the adaptive dynamics framework of Dieckmann \textit{et al.}~\cite{dieckmann2006adaptive}.
This notation matches theirs, as can be seen in equation~(\ref{eq:master}).
In the case $\mca(f)<\frac{1}{2}$, the second term of $\sigma^2$ is left out, that is, $\sigma^2(x,y)=\delta_x(y)$.
This, too, is consistent with the framework by Dieckmann \textit{et al.}~\cite{dieckmann2006adaptive}. 
They further remark that the ``boundary layer induced by inequality constraints will be very narrow whenever the canonical equation offers a valid description.'' 
Therefore, there is no smooth transition from $\sigma^2(x,y)=\delta_x(y)$ to $\sigma^2(x,y)=\delta_x(y)-12(x-\frac{1}{2})(y-\frac{1}{2})$.
The change is abrupt.
Some consequences of this will be described in Section~\ref{sec:discussion}.

The following Lemma shows that both the constrained \eqref{eq:main} and the unconstrained \eqref{eq:selection-grad} selection gradient enjoy certain mapping properties.

\begin{lemma} 
The constrained and unconstrained selection gradients, $A$ and $\nabla E$, respectively defined in \eqref{eq:def-A} and \eqref{eq:selection-grad}, are bounded operators from $L^p[0,1]$ into $L^p[0,1]$ for all $1\leq p\leq\infty$.  Moreover, they satisfy the following mapping properties: 
\begin{align}
L^p[0,1] &\mapsto W^{1,p}[0,1],\quad 1\leq p\leq \infty, \label{A-3}\\
C^k[0,1] &\mapsto C^{k+1}[0,1], \quad 0 \leq k < \infty.                      \label{A-4}
\end{align}
\label{lem:bound-p}
\end{lemma}

Here, $W^{1,p}[0,1]$ is the Sobolev space of functions in $L^p[0,1]$ such that their (weak) derivatives of first order are contained in $L^p[0,1]$.

\begin{proof}
The key to this proof is Hölder's inequality, $\|fg\|_1\leq \|f\|_p\|g\|_{p'}$, on $f\in L^p[0,1]$ and $g\in L^{p'}[0,1]$, where $1/p+1/p'=1$.
If $p=\infty$ then $p'=1$.
Consider first the unconstrained selection gradient: 
\begin{equation*}
\left|\nabla E(f)(x)\right|=\left|\int_0^x f(y)\,dy -\int_x^1 f(y)\,dy\right| = \left|\int_0^1 (\chi_{[0,x]}(y)-\chi_{[x,1]}(y))f(y)\,dy\right|.
\end{equation*}
Then apply Hölder's inequality:
\begin{multline*}
\left|\int_0^1 (\chi_{[0,x]}(y)-\chi_{[x,1]}(y))f(y)\,dy\right| \leq \|\chi_{[0,x]}-\chi_{[x,1]}\|_{p'}\|f\|_p \\
= \|f\|_p\left(\int_0^1|\chi_{[0,x]}(y)-\chi_{[x,1]}(y)|^{p'}dy\right)^{1/p'}=\|f\|_p\left(\int_0^1\,dy\right)^{1/p'}=\|f\|_p.
\end{multline*}
Therefore, $|\nabla E(f)(x)|^p \leq \|f\|_p^p$ which implies 
\begin{equation}
\label{eq:bound-E}
\int_0^1|\nabla E(f)(x)|^p dx\leq \|f\|_p^p.
\end{equation}
This is true also in the case $p=1$, since $\|\chi_{[0,x]}-\chi_{[x,1]}\|_\infty\|f\|_1=\|f\|_1.$  In case $p=\infty$ we have 
\[ |\nabla E f(x)| \leq \int_0 ^x ||f||_\infty + \int_x ^1 ||f||_\infty = ||f||_\infty \implies ||\nabla E f||_\infty \leq ||f||_\infty.\]

In order to show the same type of estimate on the constrained selection gradient we compute
\begin{multline*}
\|P\big(\nabla E(f)\big)\|_p = \left(\int_0^1|x-\tfrac{1}{2}|^{p}dx\left|12\int_0^1\big(y-\tfrac{1}{2}\big)\big(\nabla E(f)\big)dy\right|^p\right)^{1/p}\\
=12\|x-\tfrac{1}{2}\|_p\left|\int_0^1\big(y-\tfrac{1}{2}\big)\big(\nabla E(f)\big)dy\right|\\
\leq 12\|x-\tfrac{1}{2}\|_p \|x-\tfrac{1}{2}\|_{p'} \|\nabla E(f)\|_p \leq 12\|x-\tfrac{1}{2}\|_p \|x-\tfrac{1}{2}\|_{p'} \|f\|_p.
\end{multline*}
Here, we used our previous result in equation~(\ref{eq:bound-E}).
If $p=\infty$ then
\begin{multline*}
\|P\big(\nabla E(f)\big)\|_\infty=\sup_{x\in[0,1]} 12|x-\tfrac{1}{2}|\left|\int_0^1 (y-\tfrac{1}{2})\left(\int_0^y f(x)\,dx-\int_y^1 f(x)\,dx\right) dy\right| \\
=12\|x-\tfrac{1}{2}\|_\infty\left|\int_0^1 (y-\tfrac{1}{2})\left(\int_0^y f(x)\,dx-\int_x^1 f(x)\,dx\right) dy\right|\\
\leq 12\|x-\tfrac{1}{2}\|_\infty\|x-\tfrac{1}{2}\|_1\left\|\int_0^x f(y)\,dy-\int_x^1 f(y)\,dy\right\|_\infty \\
\leq 12\|x-\tfrac{1}{2}\|_\infty\|x-\tfrac{1}{2}\|_1\|f\|_\infty.
\end{multline*}
Collecting the above results, we conclude that
\begin{equation*}
\|(1-P)\nabla E(f)\|_p \leq \|\nabla E(f)\|_p+\|P(\nabla E(f))\|_p
\leq \big(1+12\|x-\tfrac{1}{2}\|_p \|x-\tfrac{1}{2}\|_{p'}\big)\|f\|_p.
\end{equation*}
This proves that $\|A(f)\|_p \leq L\|f\|_p$ where $L=\big(1+12\|x-\tfrac{1}{2}\|_p \|x-\tfrac{1}{2}\|_{p'}\big).$  

For the regularity results, equation~(\ref{eq:selection-grad}), immediately gives
\begin{equation}
\label{eq:ddxE}
\frac{d}{dx}\, \nabla E(f)(x) = 2f(x).
\end{equation}
%
This immediately implies \eqref{A-3} for the unconstrained selection gradient.  Equation~(\ref{eq:ddxE}) also shows that if $f\in C^k[0,1]$ then $\nabla E(f)\in C^{k+1}[0,1]$, which gives \eqref{A-4} for the unconstrained selection gradient.   For the constrained selection gradient we compute 
\beq \frac{d}{dx} A f(x) = 2 f(x) - 12 H(w(f)) \int_0 ^1 \left( y - \frac 1 2 \right) \left( \int_0 ^y f(z) dz - \int_y ^1 f(z) dz\right) dy. \nn 
\eeq
This implies the mapping properties \eqref{A-3} and \eqref{A-4} for the constrained selection gradient as well. 
\end{proof}

To investigate further properties of the constrained selection gradient, we begin by computing that  
\[
\sup_{x\in[0,1]} \int_0^1 |k(x,y)|\,dy \not< 1.
\]
In the sense of Kress \cite{kress2014}, $A$ is not a contraction. 
Therefore, existence and uniqueness of a solution to a Fredholm type integral equation cannot be established by Neumann series, since that would require that $A$ is a contraction. It is however a compact mapping from $L^p[0,1]$ to $L^q[0,1]$ for all $p\in(1,\infty]$ and $q\in[1,\infty)$.

\begin{proposition} \label{prop:compact} 
Both the constrained and unconstrained selection gradients are compact mappings from $L^p[0,1]$ to $L^q[0,1]$ for all $p>1,$ including $p=\infty$, and $q$ such that $1\leq q< \infty$. 
\end{proposition}

\begin{proof}
Let $(X,\mu)$ be a positive measure space, and let $k:X\times X\to\R$ be a measurable function.
For $p>1$ and $q<\infty$, define $p'=p/(1-p)$ and the ``double norm'' of $k$ by
\begin{equation*}
\|k\| = \left(\int_X\left(\int_X|k(r,s)|^{p'}\,d\mu(s)\right)^{q/p'}d\mu(r)\right)^{1/q}
\end{equation*}
or if $p=\infty$ and $q=1$, then $\|k\| =\sup\{|k(x,y)|,\ x,y\in X\}$.
If the double norm of $k$ is finite, then it defines a compact kernel operator $L^p(X)\to L^q(X)$, see Jörgens \cite{jorgens1982}, page 275--277.
The double norms of the kernels of both the unconstrained and the constrained selection gradients are finite.  
Here the space $X =[0,1]$ is the unit interval, and $\mu$ is the Lebesgue measure. 
\end{proof}

\begin{lemma}
Assume that $f$ is a measurable, bounded, non-negative function defined on the unit interval $[0,1]$, and assume that $0<\int_0^1 f(x)\,dx$. 
If $f$ is increasing (decreasing) then its MCA is bigger (smaller) than or equal to $\frac{1}{2}$.
If $f$ is continuous and strictly increasing then its MCA is strictly bigger than $\frac{1}{2}$.
\label{lem:ineq}
\end{lemma}

\begin{proof}
By the definition of MCA \eqref{eq:bm_cont_setup}, 
\[ \mca(f)<\frac{1}{2} \iff 
\int_0^1 (x-\frac{1}{2}) f(x)\,dx  < 0. \] 
By a change of variables, this is equivalent to
\[
\int_{-1/2}^{0} f\left(t+\tfrac{1}{2}\right) t \,dt + \int_{0}^{1/2} f\left(t+\tfrac{1}{2}\right) t \,dt  < 0.
\]
By another change of variables, this condition is equivalent to
\[
\int_{0}^{1/2} \left(f\left(\tfrac{1}{2}+t\right)- f\left(\tfrac{1}{2}-t\right)\right)t \,dt < 0.
\]
If $f$ is decreasing then $f(\tfrac{1}{2}+t)- f(\tfrac{1}{2}-t)\leq 0$, and if $f$ is increasing then 
\(
f(\tfrac{1}{2}+t)- f(\tfrac{1}{2}-t)\geq 0.
\)
If $f$ is continuous and strictly increasing, then $f(\tfrac{1}{2}+t)- f(\tfrac{1}{2}-t)>0$ on a set of positive measure in $[0,1/2]$.
The conclusion follows.
\end{proof}

This lemma proves that if function has ``more weight to the right'' of $x=\frac{1}{2}$ then the MCA is bigger than $\frac{1}{2}$.
We can imagine the area under the curve $y=f(x)$ as a mass distribution that balances on a tip positioned at $x=\frac{1}{2}$.
Let $f$ be a probability density function on $[0,1]$.
In particular, $\int_0^1 f(x)\,dx =1.$
The condition 
\(
\mca(f)-\frac{1}{2} = \int_0^1 (x-\frac{1}{2}) f(x)\,dx  >0
\)
means that the mass is tipping towards the right, like in the following picture:

\begin{figure}[H]
 \centering
 \begin{tikzpicture}[domain=0:4]
  \draw[->] (-0.2,0) -- (4.5,0) node[right] {$x$};
  \draw[->] (0,-0.1) node[below] {$0$} -- (0,2.1) node[right] {$y$};
  \draw     (2.0,0)  -- (2,-0.1) node[below] {$\frac{1}{2}$};
  \draw     (4.0,0)  -- (4,-0.1) node[below] {$1$};
  \fill [gray!30!white, domain=0:4, variable=\x]
      (0, 0)
      -- plot ({\x}, { 0.5*(0.1 +\x *sin(\x r) +\x)})
      -- (4, 0)
      -- cycle;
  \draw[color=blue, thick] plot (\x, { 0.5*(0.1 +\x *sin(\x r) +\x)});
  \node at (5.0,0.8) {$y=f(x)$};
 \end{tikzpicture}
 \caption{A function $f$ with $\mca(f)>\frac{1}{2}.$}
\end{figure}
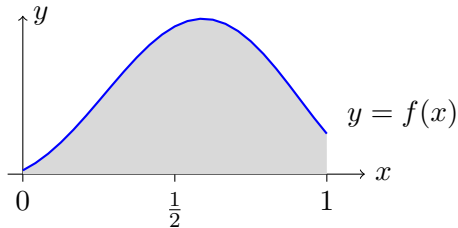

Lemma \ref{lem:ineq} reflects that $\mca(f)$ is the mean value of $f$ if we interpret $f$ as a probability density function (a mixed strategy).  Moreover, the MCA of a probability density function is likewise called the first moment about the point $x=\frac{1}{2}$.

We have thus characterized the mapping properties of the selection gradients in Lemma \ref{lem:bound-p}, and in Proposition \ref{prop:compact} we showed that they are compact mappings from $L^p[0,1]$ to itself for $1< p< \infty$.  In the following theorem we compute the spectrum of the operator $A$.  This result is important to the dynamics, since it determines the stationary solutions to the adaptive dynamics.  

\begin{theorem} 
\label{teo:eigen-f}
Let $A$ be defined by~(\ref{eq:def-A}), and fix some $1 \leq p \leq \infty$.  Let $A$ act on the elements of $L^p[0,1]$. 
Then the only solutions to the eigenvalue problem $Af=\lambda f$ are constant functions, and the corresponding eigenvalue $\lambda=0$.
\end{theorem}

\begin{proof}
We first assume that  $w(f)\geq 0$.
Then $H(w(f))=1$ in the definition~(\ref{eq:def-A}).
If $\lambda=0$, then we solve $Af=0$, which using \eqref{eq:sg_nondecreasing} is equivalent to 
\begin{equation*}
2\int_0^x f(y)\,dy - \int_0^1 f(y)\,dy 
= 12(x-\tfrac{1}{2}) \int_0^1 (y-\tfrac{1}{2}) \left( 2\int_0^y f(z)dz - \int_0 ^1 f(z) dz \right)\,dy.
\end{equation*}
The right side is a differentiable function of $x$, so the left side is also, and differentiating both sides we obtain that $f(x)$ is constant.
This completes the proof in this case. 


Assume that $\lambda\neq 0$. 
Taking the derivative of $\lambda f(x)=Af(x)$ we get 
\begin{equation*}
\lambda f'(x) = 2f(x)-12\int_0^1 \left(y-\frac{1}{2}\right)\left(\int_0^y f - \int_y^1 f\right)dy 
\end{equation*}
This is an integro-differential equation of Fredholm type with separable kernel, so the solution is found by setting the integral to some fixed real number $\beta$ and solve for $\beta$ at a later step, see Kress \cite{kress2014}. 
We obtain $\lambda f'-2f=\beta$. 
The solution is 
\begin{equation*}
f(x) = a e^{2x/\lambda} - \frac{\beta}{2}
\end{equation*}
for some constant $a$. 
Inserting this into $\lambda f=Af$, we find
\begin{multline*}
\lambda a e^{2x/\lambda} - \lambda \frac{\beta}{2} = a \frac{\lambda}{2}\left(2e^{2x/\lambda}-1-e^{2/\lambda}\right)
\\
-12\left(x-\frac{1}{2}\right)\int_0^1 \left(y-\frac{1}{2}\right)\left(a\frac{\lambda}{2}\left(2e^{2y/\lambda}-1-e^{2/\lambda}\right)\right)dy,
\end{multline*}
so canceling $\lambda$  and subtracting $ae^{2x/\lambda}$ on both sides,
\begin{equation*}
- \frac{\beta}{2} = a\left(-1-e^{2/\lambda}\right)-12a\left(x-\frac{1}{2}\right)\int_0^1 \left(y-\frac{1}{2}\right)\left(e^{2y/\lambda}-\frac{1}{2}-\frac{1}{2}e^{2/\lambda}\right)dy.
\end{equation*}
Since $-1-e^{2/\lambda}$ is constant, 
\[
\int_0^1 \left(y-\frac{1}{2}\right)\left(-\frac{1}{2}-\frac{1}{2}e^{2/\lambda}\right)dy = 0.
\]
Furthermore, since $e^{2y/\lambda}$ is either increasing or decreasing (but never constant),
we may apply Lemma~\ref{lem:ineq} and obtain
\[
\int_0^1 \left(y-\frac{1}{2}\right)e^{2y/\lambda}dy = K \neq 0,
\]
where $K=\frac{\lambda}{2}\left(\frac{1}{2}-\frac{\lambda}{2}\right)e^{2/\lambda}+\frac{\lambda}{2}\left(\frac{1}{2}+\frac{\lambda}{2}\right)$ is a constant.  
Thus, 
\begin{equation*}
- \frac{\beta}{2} = a\left(-1-2e^{2/\lambda}\right)-12aK\left(x-\tfrac{1}{2}\right).
\end{equation*}
The left hand side is a constant but the right hand side varies linearly with $x$.
Therefore, it must be that $a=0$ and $\beta=0$. 
In the case $\mca(f)<\frac{1}{2}$, an almost identical computation proves that the only solution to $Af=\lambda f$ is $f=0$.  
\end{proof}

Although we will not make much use of the following result, it is interesting due to the connection it shows between the adaptive dynamics on the function valued game and the discrete game. 
In particular, these games share a certain symmetry.
If $A$ is given by~(\ref{eq:def-A}), then its kernel is anti-symmetric about the point $(x,y)=(\frac{1}{2},\frac{1}{2})$,
\[
k(1-x,1-y)= -s(x,y)-12H(w(f))(x-\tfrac{1}{2})(y^2-y)=-k(x,y).
\]
The skew symmetry of $A$ will be carried over to the discrete game in Section~\ref{sec:computations}. 
In the discrete game dynamics, $A$ is a $n\times n$ matrix which is \textit{skew centro-symmetric}, which means that $A$ is skew symmetric with respect to the intersection of its primary and its secondary diagonal. 
This means that we may say that $A$ -- regardless of whether it applies to the vector-valued strategies or the function-valued ones -- is skew centro-symmetric.


The $L^2$ adjoint $A^*$ is defined by $\langle Au,v\rangle=\langle u,A^*v\rangle$. 
A straightforward computation (integration by parts) gives the following lemma.

\begin{lemma}
Consider function-valued strategies with mean value $\frac{1}{2}.$
Let $A:L^2[0,1]\to L^2[0,1]$ be given by~(\ref{eq:def-A}).
The $L^2$ adjoint of $A$ is given by
\begin{equation*}
A^*v(x) = -\left(\int_0^x v(y)\,dy-\int_x^1 v(y)\,dy\right) + 12\left(x^2-x\right)\int_0^1\left(y-\frac{1}{2}\right) v(y)\,dy.
\end{equation*}
Let $k$ be the kernel of $A$ and $k^*$ the kernel of $A^*$.
If $T$ is the transformation of coordinates that mirrors $(x,y)$ in the point $(0.5,0.5)$, then $Tk=-k^*$.
\end{lemma}

\subsection{Existence of dynamical solutions} \label{sec:existence} 

Here we investigate solutions to the initial value problem for the constrained adaptive dynamics \eqref{eq:main}, with constrained selection gradient \eqref{eq:def-A}, and with initial data contained in $L^p[0,1]$ for $1 \leq p \leq \infty$.  
For such initial data $f_0$,  $\alpha$ is the integral curve of $A$ starting at $f_0$ if 
\begin{equation}
\frac{d}{dt} \alpha(t)=A\big(\alpha(t)\big)\quad\text{ with }\quad\alpha(0)=f_0.
\label{eq:problem}
\end{equation}
We will show that there is a family of integral curves that depends continuously on the initial strategy.  An important question about the problem~(\ref{eq:problem}) is whether there exist stationary solutions, that is, functions $f$ such that $\frac{d}{dt}f=Af=0.$ 
We answer this question in the following proposition.
It is interesting to note that in the context of the function-valued team game, the stationary solutions are precisely the equilibrium strategies of the game.

\begin{proposition}
\label{prop:stationary}
The only solutions to the equation for stationary solutions, $\frac{d}{dt}f=Af=0$, with initial data in $L^p[0,1]$ for some $ 1 \leq p \leq \infty$ are constant functions.
Specifying to functions satisfying \eqref{eq:bm_cont_setup}, the only stationary solutions to \eqref{eq:problem} are precisely the equilibrium strategies of the function-valued team game given in Theorem \ref{th:eqstr_function}.
\end{proposition}

\begin{proof}
By Theorem~\ref{teo:eigen-f}, the equation $Af=\lambda f$ has only one solution $f$ such that $f$ is not the zero function, namely $f$ being a constant on $[0,1]$.
The corresponding  eigenvalue is $\lambda=0$ and therefore, the equation for stationary solutions is satisfied. 
\end{proof}

\begin{remark}
\label{rem:constant}
If the strategies are normalized, they can be interpreted as probability density functions.
It turns out that strategies $f$ that satisfy $\mca(f)=\frac{1}{2}$ stay normalized during the adaptive dynamics evolution.
For any function $f$, 
\begin{equation*}
\int_0^1 P(f)(x)\,dx=0,
\end{equation*}
since $P(f)(x)$, which is defined in~(\ref{eq:pi}), is proportional to $x-\frac{1}{2}$.
By integration by parts,
\begin{equation}
\int_0^1 \left(\int_0^x f(y)\,dy - \int_x^1 f(y)\,dy\right)\,dx= 2\left(\frac{1}{2}-\mca(f)\right)\int_0^1 f(x)\,dx.
\label{eq:int-grad}
\end{equation}
That is,
\begin{equation*}
\text{MCA}(f)=\frac{1}{2}\implies \int_0^1 \nabla E(f)(x)\,dx=0.
\end{equation*}
Therefore, if the initial data $f_0$ is such that $\mca(f_0)=\frac{1}{2}$, and $f$ evolves according to 
\[ \frac{\partial f}{\partial t} = (1-P)\nabla E(f) = A f,\] then the integral $\int_0^1 f(x)\,dx$ is constant. 
In other words, the population size is preserved. 
When $f$ represents a probability distribution function it stays normalized during the evolution. 
\end{remark}

Notice that equation~(\ref{eq:problem}) can be cast into the integral form
\begin{equation}
\alpha(t) = f_0 + \int_0^t A\big(\alpha(s)\big)\, ds.
\label{eq:integral-main}
\end{equation}
This is no longer an integro-differential equation but rather an integral equation of mixed Volterra and Fredholm type.
The existence of solutions to the initial value problem is transferred into a fixedpoint problem of~(\ref{eq:integral-main}).
In this situation it is suitable to use the Banach fixed point theorem, or even better, the Picard-Lindelöf theorem.
We will apply these in Theorem~\ref{Lp-theorem}.

If $\dot \alpha$ is the derivative of $\alpha$ with respect to time, equation~(\ref{eq:problem}) can be written as $\dot \alpha=A(\alpha)$.
Here, we mean that $A(\alpha)$ is the image of $A$ applied to $\alpha(t)$, which is a function in $L^p[0,1]$.
In order to clarify this, we use the notation $\alpha_t=\alpha(t)$ for the integral curve at time $t$.

\begin{lemma}
Consider the problem~(\ref{eq:problem}) with $A$ defined by~(\ref{eq:def-A}).
Assume that it admits a local solution, $\alpha_t$, for $t$ in some interval $J$. If the solution satisfies $\int_0^1 \alpha_t(x)\,dx\neq 0$ at some $t\in J$, and $\mca(\alpha_t)=\frac{1}{2}$ 
then the time derivative of $\mca(\alpha)$ at time $t$ vanishes. If $0<\mca(\alpha_t)<\frac{1}{2}$, and the solution satisfies $\alpha_t(x)>0$ at some $t\in J$  then $\mca(\alpha_t)$ is an increasing function of $t$ with growth rate $\frac{d}{dt}\mca(\alpha_t)>2(1/2-\mca(\alpha_t))^2.$
\label{lem:increase-mca}
\end{lemma}

\begin{proof}
Let $\alpha$ be a solution to~(\ref{eq:problem}) with $A$ defined by~(\ref{eq:def-A})
and denote by $\alpha_t$ the solution after time $t$.
First, note that since $\alpha_t$ is a solution it is measurable, and $\alpha$ is a piecewise $C^1$-curve on $L^p[0,1]$, where $p$ is the same as for the initial data.
Consider the case $w(\alpha_t)<0$.
Then $A=\nabla E$, which satisfies a Lipschitz condition on the $L^p[0,1]$ space.
Thus \cite{lang1993}, we may compute
\begin{equation*}
\frac{d}{dt} \int_0^1 \alpha_t(x)dx = \int_0^1 (A\alpha_t)(x)\,dx.
\end{equation*}
Using $A=\nabla E$,
\begin{equation*}
\int_0^1 (A\alpha_t)(x)\,dx = \int_0^1 \left(\int_0^x\alpha_t(y)dy - \int_x^1\alpha_t(y)dy\right)\,dx.
\end{equation*}
The right hand side of this is simplified by~(\ref{eq:int-grad}) such that
\begin{equation*}
\frac{d}{dt} \int_0^1 \alpha_t(x)dx = 2\left(\frac{1}{2}-\mca(\alpha_t)\right) \int_0^1\alpha_t(x)\,dx,
\end{equation*}
so the denominator in the definition of the MCA, see \eqref{eq:bm_cont_setup}, does not approach zero.
Moreover,
\begin{multline*}
\frac{\partial}{\partial t} \int_0^1 x\alpha_t(x)\,dx = \int_0^1 x\left(\int_0^x\alpha_t(y)dy - \int_x^1\alpha_t(y)dy\right) dx\\
= \int_0^1 \left(x-\frac{1}{2}\right)\left(\int_0^x\alpha_t(y)dy - \int_x^1\alpha_t(y)dy\right) dx
+\frac{1}{2}\int_0^1 \left(\int_0^x\alpha_t(y)dy- \int_x^1\alpha_t(y)dy\right) dx.
\end{multline*}
The first term in this sum is simplified using integration by parts:
\begin{equation*}
 \int_0^1 \left(x-\frac{1}{2}\right)\left(\int_0^x\alpha_t(y)dy - \int_x^1\alpha_t(y)dy\right) dx
 = -\int_0^1 (x^2-x)\alpha_t(x)\, dx.
\end{equation*}
If $\alpha_t$ is a strategy of the team game, then according to the assumptions \eqref{eq:bm_cont_setup} 
\begin{equation*}
 -\int_0^1 (x^2-x)\alpha_t(x)\, dx = \int_0^1 (x-x^2)\alpha_t(x)\, dx \geq 0.
\end{equation*}
Collecting these results we obtain
\begin{multline*}
\frac{d}{dt}\mca(\alpha_t) = \left(\frac{d}{dt}\int_0^1 x\alpha_t(x)dx\right)\frac{1}{\int_0^1 \alpha_t(x)dx} \\
- \left(\int_0^1 x\alpha_t(x)dx\right)\frac{1}{\left(\int_0^1 \alpha_t(x)dx\right)^2}\left(\frac{d}{dt}\int_0^1 \alpha_t(x)dx\right)\\
=\left(\int_0^1 (x-x^2)\alpha_t(x)dx +\left(\frac{1}{2}-\mca(\alpha_t)\right)\int_0^1\alpha_t(x)dx \right)\frac{1}{\int_0^1 \alpha_t(x)dx} \\
- \left(\int_0^1 x\alpha_t(x)dx\right)\frac{1}{\left(\int_0^1 \alpha_t(x)dx\right)^2}2\left(\frac{1}{2}-\mca(\alpha_t)\right)
\end{multline*}
\begin{multline*}
=\frac{1}{\int_0^1 \alpha_t(x)\,dx}\int_0^1 (x-x^2)\alpha_t(x)dx +\left(\frac{1}{2}-\mca(\alpha_t)\right)
-2\mca(\alpha_t)\left(\frac{1}{2}-\mca(\alpha_t)\right)\\
=2\left(\frac{1}{2}-\mca(\alpha_t)\right)^2+\frac{1}{\int_0^1 \alpha_t(x)\,dx}\int_0^1 (x-x^2)\alpha_t(x)dx.
\end{multline*}
This shows that the MCA of a strategy $f$ such that $0<\mca(f)<\frac{1}{2}$ has a positive derivative with respect to time in the adaptive dynamics system.

Next, consider the case $w(\alpha_t)\geq 0$.
Define $A_c$ by
\begin{equation*}
A_cf(x)=\int_0^x f(y)\,dy- \int_x^1 f(y)\,dy -12(x-\tfrac{1}{2})\int_0^1 (y-\tfrac{1}{2})\left(\int_0^y f - \int_y^1 f\right)dy.
\end{equation*}
Since $w(\alpha_t)\geq 0$, the operators $A$ and $A_c$ coincide.
Notice that $A_c$ is linear, and it is bounded on $L^p[0,1]$ by Lemma~\ref{lem:bound-p}.
Again we may differentiate under the integral with respect to $t$.
Since $\int_0^1(x-\frac{1}{2})dx = 0$, 
\begin{equation}
\frac{d}{dt} \int_0^1 \alpha_t(x)dx = \int_0^1 \left(\int_0^x\alpha_t(y)dy - \int_x^1\alpha_t(y)dy\right)\,dx.
\label{eq:dda1}
\end{equation}
As in the previous case, the denominator of the MCA does not approach zero whenever $\alpha_t$ is a strategy.
Moreover, by the definition of $A_c$, 
\begin{multline*}
\frac{d}{dt} \int_0^1 x\alpha_t(x)\,dx = \int_0^1 x\left(\int_0^x\alpha_t(y)dy - \int_x^1\alpha_t(y)dy\right.\\
\left. -12 (x-\tfrac{1}{2})\int_0^1(y-\tfrac{1}{2})\left(\int_0^y \alpha_t(z)dz - \int_y^1 \alpha_t(z)dz\right)dy\right) dx.
\end{multline*}
Let $q_t=\int_0^1(y-\tfrac{1}{2})\left(\int_0^y \alpha_t(z)dz - \int_y^1 \alpha_t(z)dz\right)dy$ and compute
\begin{multline*}
\int_0^1 x\left(\int_0^x\alpha_t(y)dy - \int_x^1\alpha_t(y)dy-12q_t(x-\tfrac{1}{2})\right) dx\\
= \int_0^1 \left(x-\frac{1}{2}\right)\left(\int_0^x\alpha_t(y)dy - \int_x^1\alpha_t(y)dy-12q_t(x-\tfrac{1}{2})\right) dx 
\\+\frac{1}{2} \int_0^1 \left(\int_0^x\alpha_t(y)dy - \int_x^1\alpha_t(y)dy-12q_t(x-\tfrac{1}{2})\right) dx.
\end{multline*}
Thus,
\begin{multline}
\frac{d}{dt} \int_0^1 x\alpha_t(x)\,dx  = \int_0^1 \left(x-\frac{1}{2}\right)\left(\int_0^x\alpha_t(y)dy - \int_x^1\alpha_t(y)dy\right) dx\\
-12\underbrace{\int_0^1(x-\tfrac{1}{2})^2\,dx}_{=1/12}\int_0^1 \left(x-\frac{1}{2}\right)\left(\int_0^x\alpha_t(y)dy - \int_x^1\alpha_t(y)dy\right)dx 
\\+\frac{1}{2} \int_0^1 \left(\int_0^x\alpha_t(y)dy - \int_x^1\alpha_t(y)dy-12q_t(x-\tfrac{1}{2})\right) dx\\
= \frac{1}{2} \int_0^1 \left(\int_0^x\alpha_t(y)dy - \int_x^1\alpha_t(y)dy\right) dx.
\label{eq:dda2}
\end{multline}
The results in~(\ref{eq:dda1}) and~(\ref{eq:dda2}) together implies that
\[
 \frac{d}{dt} \int_0^1 \left(x-\tfrac{1}{2}\right)\alpha_t(x)\, dx = 0.
\]
In other words,
\[
 w(\alpha_t)\geq 0 \implies \frac{d}{dt} w(\alpha_t) = 0.
\]
The denominator of the MCA is not approaching zero, so this implies that the derivative of $\mca(\alpha_t)$ with respect to $t$ is zero.

Computing $\frac{d}{dt}\mca(\alpha_t)$ in the case $w(\alpha_t)\geq 0$ is just like in the case $w(\alpha_t)<0$ but with an additional term. 
The result for both cases together is
\begin{equation}
\label{eq:mca-growth}
\frac{d}{dt}\mca(\alpha_t) = 2\left(\frac{1}{2}-\mca(\alpha_t)\right)^2+\frac{1-H(w(\alpha_t))}{\int_0^1\alpha_t(x)dx}\int_0^1x(1-x)\alpha_t(x)dx.
\end{equation}
In particular,
\[
 \mca(\alpha_t)=1/2\ \implies \frac{d}{dt}\mca(\alpha_t) =0.
\]
If $\alpha_t$ is a strategy fulfilling \eqref{eq:bm_cont_setup}, then the second term on the right side of \eqref{eq:mca-growth} is non-negative, so if also $0<\mca(\alpha_t)<\frac{1}{2}$ then $\frac{d}{dt}\mca(\alpha_t)>0$.
Removing the second part of the RHS in~(\ref{eq:mca-growth}), a lower bound on $\mca(\alpha_t)$ is obtained. 
\end{proof}

\begin{remark}
We remark that the lower bound of the above lemma approaches $\frac{1}{2}$ as $t\to\infty.$
We have $\frac{\partial}{\partial t}\mca(\alpha_t)>2(1/2-\mca(\alpha_t))^2$.
Consider a function $g(t)$ that grows exactly according to $g'=2(1/2-g)^2$.
Solving $g'/(g-1/2)^2 = 2$ yields
\[
g(t) = \frac{1}{2} -\frac{1}{c_0+2t}.
\]
Here, $c_0$ is a constant of integration.
We make sure that $g(0)=\mca(f_0)$ by solving $\mca(f_0)=\frac{1}{2}-\frac{1}{c_0}$, where $f_0$ is the initial data. 
Since $\mca(f_0)<\frac{1}{2}$ we obtain $c_0>0$.
For any $t\geq 0$, the MCA is increasing, and 
\[
\mca(\alpha_t)>\frac{1}{2} -\frac{1}{c_0+2t}\quad\text{ for } t>0.
\]
The lower bound in this equation approaches $\frac{1}{2}$ as $t\to\infty$.
\end{remark}

This lemma will be used in the following theorem, which is one our central results.

\begin{theorem} 
\label{Lp-theorem}
Fix $p$ with $1 \leq p \leq \infty$.
Let $A$ be as in~(\ref{eq:def-A}) and $f_0\in L^p[0,1]$.
Then the initial value problem~(\ref{eq:problem}) admits a solution $\alpha:[0,\infty)\to L^p[0,1]$.
If in addition $\int_0^1 f_0(x)\,dx\neq 0$ and $\mca(f_0)=\frac{1}{2}$ then $\mca(\alpha(t))=\frac{1}{2}$ for all $t>0$.
If furthermore $f_0\in L^p[0,1]$ with $p\geq 2$ and $\mca(f_0)=\frac{1}{2}$ then the $L^2$ norm of the solution is constant for all $t>0.$
If $f_0\in C^k([0,1])$ then the solution is also $C^k$ at every time.
\end{theorem}

\begin{proof}
The operator $A$ \eqref{eq:def-A} changes with $w$ \eqref{eq:w} which is used to express the MCA constraint.  We analyze the situations $w(f_0)<0$ and $w(f_0)\geq 0$ separately.
Starting with the latter, let us define $A_c:L^p[0,1]\to L^p[0,1]$ by
\begin{equation*}
A_cf(x)=\int_0^x f(y)\,dy- \int_x^1 f(y)\,dy -12(x-\tfrac{1}{2})\int_0^1 (y-\tfrac{1}{2})\left(\int_0^y f - \int_y^1 f\right)dy.
\end{equation*}
Notice that $A_c$ is linear on $L^p[0,1]$, and by applying Lemma~\ref{lem:bound-p} to $A_c$ we obtain a Lipschitz condition: $\|A_c(f)-A_c(g)\|\leq L\|f-g\|$ for $L=(1+12\|x-\frac{1}{2}\|_p\|x-\frac{1}{2}\|_q)$ and for all $f,g\in L^p[0,1]$.
Therefore, the Picard-Lindelöf theorem applies.
We quote the theorem as stated by Brezis \cite{brezis2011}:
\begin{quote}
\textbf{Theorem 7.3 of \cite{brezis2011}.} \em Let $E$ be a Banach space with norm $\|\cdot\|$ and $F:E\to E$ a Lipschitz mapping, i.e., there is a constant $L$ such that $\|Fu-Fv\|\leq L\|u-v\|$ for all $u,v\in E$.
Given $u_0\in E$, there is a unique $C^1$-curve \[u:[0,\infty)\to E\] satisfying the initial value problem \[ du/dt = F(u),\quad u(0)=u_0.\] \em 
\end{quote}
Here, the mapping $F$ corresponds to $A_c$, $L=(1+12\|x-\frac{1}{2}\|_p\|x-\frac{1}{2}\|_q)$ and $E=L^p[0,1]$.
Since the Lipschitz condition on $A_c$ is independent of $f_0$ and global on $L^p[0,1]$, the solution $\alpha$ to the initial value problem $\dot \alpha=A_c\alpha,\ \alpha(0)=f_0$ is defined on $t>0$ for any given initial data $f_0\in L^p[0,1]$.

Next, if $w(f_0)<0$ then $A=\nabla E$, so we would like to analyze the mapping $\nabla E :L^p[0,1]\to L^p[0,1]$ given by~(\ref{eq:selection-grad}).
Since $\nabla E$ is linear on $L^p[0,1]$ and bounded (by Lemma~\ref{lem:bound-p}) we obtain a Lipschitz condition: $\|\nabla E(f)-\nabla E(g)\|_p \leq \|f-g\|_p.$
Again, the Picard-Lindelöf theorem applies and we obtain existence and uniqueness of solutions, but this time the solution satisfies $\dot \alpha=\nabla E(\alpha),\ \alpha(0)=f_0$.

Following the same computations as in the proof of Lemma~\ref{lem:increase-mca},
\begin{equation*}
w(\alpha_t)\geq 0\implies \frac{\partial}{\partial t} w(\alpha_t)=
\frac{\partial}{\partial t} \int_0^1 (x-\tfrac{1}{2}) \alpha_t(x)dx = 0.
\end{equation*}
That is, if the solution is such that $w(\alpha_t)\geq 0$ at some $t$ then it continues to be such that $w(\alpha_t)\geq 0$.
If on the other hand $w(\alpha_t)<0$ then it might happen that $w(\alpha_{t_1})=0$ at some later time $t_1$.
Then we may solve $\dot \alpha=A\alpha$ on $t>t_1$ with $\alpha_{t_1}$ as initial condition to obtain a unique solution at later times.
In conclusion, the initial value problem~(\ref{eq:problem}) admits a solution $\alpha:[0,\infty)\to L^p[0,1]$ for any initial data $f_0\in L^p[0,1]$.

As in the beginning of the proof of Lemma~\ref{lem:increase-mca}, we can conclude that if $\int_0^1 f_0(x)\,dx\neq 0$ for any initial data $f_0\in L^p[0,1]$ then $\int_0^1 \alpha_t(x)\,dx\neq 0$ for all $t.$
Assuming that the initial data $f_0$ satisfies $\mca(f_0)=1/2$, Lemma~\ref{lem:increase-mca} implies that the MCA of the solution curve $\alpha(t)$ is constant for all $t>0.$
Thus also the initial value problem with $A$ instead of $A_c$, that is, $\dot \alpha=A\alpha,\ \alpha(0)=f_0$ with $\mca(f_0)=\frac{1}{2}$ admits a solution with constant MCA.

Next, we prove that if $f_0$ is $C^k$-smooth then the solution $\alpha(t)$ is also $C^k$-smooth with respect to $x$.  We first compute that the derivative of $A\alpha_t(x)$ with respect to $x$ is
\begin{equation*}
\frac{\partial}{\partial x} A\alpha_t(x) = 2\alpha_t(x) -12 H(w(\alpha_t))\int_0^1 \left(y-\frac{1}{2}\right)\left(\int_0^y \alpha_t(z)\,dz-\int_y^1 \alpha_t(z)\,dz\right)dy.
\end{equation*}

If $\alpha_t$ is $C^k$-smooth with respect to $x$, then $A\alpha_t$ is $C^{k+1}$-smooth but then $\alpha_t$ is $C^{k+1}$-smooth. Using the integral form of the dynamics equation \eqref{eq:integral-main}, it follows that the initial data $f_0$ determines the smoothness of $\alpha(t)$ with respect to $x$.

Let $f_0\in L^p[0,1]$ with $p\geq 2$.
To show that the solution's $L^2$ norm is constant, let $p=2$ and compute the time derivative of the $L^2$ norm 
\[
\frac{d}{dt}\|\alpha(t)\|^2 = \frac{d}{dt}\langle \alpha(t),\alpha(t)\rangle = 2\langle \dot{\alpha}(t),\alpha(t)\rangle=2\langle A\alpha(t),\alpha(t)\rangle.
\]
Let $q_0=\int_0^1 (y-\frac{1}{2})(\int_0^y \alpha_t(x)\,dx-\int_y^1 \alpha_t(x)\,dx)dy$.
Then
\begin{multline*}
\frac{d}{dt}\|\alpha(t)\|^2 = 2\int_0^1 \alpha_t(x) \left(\int_0^x \alpha_t(x)\,dx -\int_x^1 \alpha_t(x)\,dx - 12q_0 H(w(f))(x-\tfrac{1}{2})\right)\,dx\\
=2\underbrace{E[\alpha(t),\alpha(t)]}_{=0}-24H(w(f))q_0\int_0^1(x-\tfrac{1}{2})\alpha_t(x)\,dx
\end{multline*}
\begin{equation}
\label{eq:ddt-alpha}
\implies \frac{d}{dt}\|\alpha(t)\| = \frac{-12q_0 H(w(f))\int_0^1(x-\tfrac{1}{2})\alpha_t(x)\,dx}{\|\alpha(t)\|}.
\end{equation}
Here, we used that $E[f,f]=0$ for all $f\in L^1[0,1]\supset L^p[0,1]$.
Notice that $q_0\geq 0$ by~(\ref{eq:ibp-grad}).
Recall from the first part of the proof that if $\mca(f_0)=\frac{1}{2}$ then also $\mca(\alpha_t)=\frac{1}{2}.$
Since $\int_0^1(x-\tfrac{1}{2})\alpha_t(x)\,dx=0$ is equivalent to $\mca(\alpha_t)=\frac{1}{2}$, a consequence of~(\ref{eq:ddt-alpha}) is
\begin{equation*}
\frac{d}{dt}\|\alpha(t)\| = 0.
\end{equation*}
That is, the solution $\alpha(t)$ has constant $L^2$ norm.
\end{proof}

At this point, there is no guarantee that the solution to the problem~(\ref{eq:problem}) improves the strategy. 
Given a solution $\alpha(t)$ with $\alpha(0)=f_0$, we say that $\alpha_t$ defeats $f_0$ if $E[\alpha(t),f_0]> 0$ for some $t>0$.
If $f_0$ is a constant function, Rowlett \textit{et al.}~\cite{rowlett2022diversity} showed that it does not exist strategies with positive expectation in competition with $f_0$, but in case $f_0$ is non-constant we would like to know if the evolved strategy improves in the sense that it would defeat $f_0$.
That is, we would like to analyze whether $E[\alpha(t),f_0]$ is positive or negative for some $t>0$.

\begin{lemma}
Let $f_0\in\mathcal{L}^\infty_+[0,1],$ and assume that $f_0$ is not a constant function.
Then, for small $t$, the solution to the adaptive dynamics equation~(\ref{eq:problem}) defeats $f_0$, or in other words $E[\alpha(t),f_0]>0$.
\label{lem:nonconstant}
\end{lemma}

\begin{proof}
First consider the case $\mca(f_0)<\frac{1}{2}$.
Since $f_0$ is non-negative and not constant, $\int_0^1 f_0(x)\,dx>0$.
The solution curve $\alpha$ is continuous, so for sufficiently small $t$, $\int_0^1 \alpha_t(x)\,dx>0$.
The constraint $\mca(f_0)<\frac{1}{2}$ is equivalent to 
\[
\int_0^1 (x-\tfrac{1}{2}) f_0(x)\,dx<0,
\]
and again, for small $t$, this holds also for $\alpha_t$ by the continuity of the solution curve.
That is, there exist a constant $b>0$ and an open interval $-b<t<b$ such that $\mca(\alpha_t)<\frac{1}{2}$ for all $t\in(-b,b)$.
Then,
\begin{align*}
\frac{d}{dt} E[\alpha_t,f_0] \,
&= \int_0^1 A\alpha_t (x)\left(\int_0^x f_0(y)\,dy - \int_x^1 f_0(y)\,dy\right)dx\\
&= \int_0^1 \left(\int_0^x \alpha_t (y)\,dy-\int_x^1 \alpha_t (y)\,dy\right)\left(\int_0^x f_0(y)\,dy - \int_x^1 f_0(y)\,dy\right)dx\\
&=\langle \nabla E(\alpha_t), \nabla E(f_0)\rangle.
\end{align*}
Thus,
\begin{equation}
\left.\frac{d}{dt} E[\alpha_t,f_0]\right|_{t=0}
=\langle \nabla E(f_0), \nabla E(f_0)\rangle>0.
\label{eq:de-positive}
\end{equation}
Consider $t$ in the interval $-b<t<b$ such that $\mca(\alpha_t)<\frac{1}{2}$.
Then $A(\alpha_t)=\nabla E(\alpha_t)$.
The function $t\mapsto \langle \nabla E(\alpha_t), \nabla E(f_0)\rangle$ is continuous, since the inner product $\langle\,,\rangle$ is continuous and $\nabla E$ is Lipschitz continuous.
Since the derivative of $E[\alpha(t),f_0]$ is positive at $t=0$, by~(\ref{eq:de-positive}), and continuous, there exists a time $t>0$ such that $E[\alpha(t),f_0]>0.$

Second, consider the case $\mca(f_0)=\frac{1}{2}$.
Then the selection gradient is $A(f_0)=(1-P)\nabla E(f_0)$.
By construction, then, $\mca(\alpha_t)=\frac{1}{2}$ for all $t\geq 0$.
Then,
\begin{multline*}
\frac{d}{dt} E[\alpha_t,f_0]
= \int_0^1 A\alpha_t (x)\left(\int_0^x f_0(y)\,dy - \int_x^1 f_0(y)\,dy\right)dx\\
= \int_0^1 \left(\int_0^x \alpha_t (y)\,dy-\int_x^1 \alpha_t (y)\,dy\right)\left(\int_0^x f_0(y)\,dy - \int_x^1 f_0(y)\,dy\right)dx\\
\quad -12\int_0^1(x-\tfrac{1}{2})\alpha_t(x) dx \int_0^1(x-\tfrac{1}{2}) \left(\int_0^x f_0(y)\,dy - \int_x^1 f_0(y)\,dy\right) dx\\
 = \langle A(\alpha(t)),\nabla E(f_0)\rangle
\end{multline*}
Since $\mca(\alpha_t)=\frac{1}{2}$ for all at all positive times, $\|A(\alpha_{t}-\alpha_{t'})\| < K\|\alpha_{t}-\alpha_{t'}\|$ for $t,t'>0$ for some $K>0$ by Lemma~\ref{lem:bound-p}.
It follows that $t\mapsto \langle A(\alpha(t)),\nabla E(f_0)\rangle$ is continuous.
At $t=0$ we have $\alpha(0)=f_0$ and 
\begin{equation*}
E[A(f_0),f_0] = \langle A(f_0),\nabla E(f_0)\rangle 
=\frac{1}{\|\tfrac{1}{2}-x\|^2}\left( \|\tfrac{1}{2}-x\|^2\|\nabla E(f_0)\|^2-\langle \tfrac{1}{2}-x,\nabla E(f_0)\rangle^2 \right)
\end{equation*}
By the Cauchy-Schwarz inequality,
\begin{equation}
\label{eq:CS-ineq}
\|\tfrac{1}{2}-x\|^2\|\nabla E(f_0)\|^2\geq \langle \tfrac{1}{2}-x,\nabla E(f_0)\rangle^2
\end{equation}
with equality if and only if $a(\frac{1}{2}-x)=b\nabla E(f_0)(x)$ for non-zero constants $a,b$.
Hence by the Cauchy-Schwarz inequality, $E[A(f_0),f_0]$ is positive or zero, and it is zero if and only if $a(\frac{1}{2}-x)=b\nabla E(f_0)(x)$. 
Since $a(\frac{1}{2}-x)=b\nabla E(f)(x)$ only if $f$ is constant almost everywhere, the inequality~(\ref{eq:CS-ineq}) is strict whenever $f_0$ is not constant.
\end{proof}

The adaptive dynamics solutions are such that a constant function is the only function that is not defeated by the selection gradient, but as the next proposition will show, it can be difficult to find a solution that becomes a constant function at some point in time.

\begin{proposition}
\label{prop:never-const}
Assume that the initial condition of the problem~(\ref{eq:problem}) is such that $\mca(f_0)=\frac{1}{2}$ and that $f_0$ is not constant.
Then the solution to this problem, $\alpha(t)$, is never the constant function.
\end{proposition}

\begin{proof}
If $u$ is the constant function, without loss of generality assume $u=1$, then for any function $g$ in $L^p[0,1]$ for $1 \leq p \leq \infty$,
\begin{equation}
E[u,g] = 2\left(\tfrac{1}{2}-\mca(g)\right)\int_0^1g(x)\,dx.
\label{eq:E-of-u}
\end{equation}
Let $\alpha$ be the integral curve of $A$ with $\alpha(0)=f_0$. 
Since $f_0$ is not constant, for sufficiently small $t>0$,
\[
E[\alpha(t),f_0]>0
\]
by Lemma~\ref{lem:nonconstant}.
This is impossible if $\alpha(t)$ is the constant function by equation~(\ref{eq:E-of-u}), since $\mca(f_0)=\frac{1}{2}$.
\end{proof}

Proposition~\ref{prop:never-const} says that if the MCA of the initial data equals $\frac{1}{2}$ then the stationary solution is either existing from the start or it is never reached.

\subsection{Implications for the evolution of strategies in the function-valued team game} 
\label{sec:game-dynamics}

Whereas the adaptive dynamics was analyzed on $L^p[0,1]$ in Section~\ref{sec:existence}, this section describes the adaptive dynamics of function-valued strategies which imposes further restrictions on the initial data for the adaptive dynamics.  

One challenge in dealing with strategies is coming from the non-negativity of strategies. 
Recall that if $f$ is a strategy, then
\begin{equation}
\label{eq:int-f}
\int_a^b f(x)\,dx
\end{equation}
represents the number of individuals in the species $f$ with competitive ability between $a$ and $b$. 
Thus, the quantity~(\ref{eq:int-f}) should be a non-negative number for any choice of interval $(a,b) \subset [0,1]$.
The adaptive dynamics might try to evolve a strategy out of the strategy space, for example by breaking the non-negativity condition. 

\begin{theorem}
\label{teo:strategy}
Let $f_0$ be either continuous or an element of $L^\infty[0,1]$. Assume that it satisfies \eqref{eq:bm_cont_setup} and $f_0(x)\geq K$ for some $K>0$ and for all $x\in[0,1]$. Then for some $T>0$ there exists a solution $\alpha$ to \eqref{eq:problem} that satisfies \eqref{eq:bm_cont_setup} for $t\in[0,T]$ and is respectively, continuous or an element of $L^\infty[0,1]$ for all $t \in [0,T]$.
\end{theorem}

\begin{proof}
First, observe that continuous functions on $[0,1]$ are all contained in $L^\infty[0,1]$.  The problem~(\ref{eq:problem}) admits a solution $\alpha:[0,\infty)\to L^\infty[0,1]$ 
by Theorem~\ref{Lp-theorem}, which is as smooth as the initial data. 
Thus $\alpha(t)-f_0$ is well-defined in $L^\infty[0,1]$ for all $t\in[0,T]$ for some $T>0.$
Define
\[
\bar{M} = \sup_{t\in[0,T]} \|\alpha(t)-f_0\|_\infty = \sup_{t\in[0,T]} \sup_{x\in[0,1]}|\alpha(t)(x)-f_0(x)|.
\]
Then,
\[
\|\alpha(t)-f_0\|_\infty=\left\|\int_0^t A\alpha(s)\,ds\right\|_\infty\leq t\bar{M}\quad\text{ for all } t<T.
\]
This equation shows that if $t$ is sufficiently small, then $\alpha(t)$ is sufficiently close to $f_0$ in the supremum norm. 
Therefore, if $f_0$ is such that $f_0(x)\geq K$ for all $x$ and for some $K>0$, we fix a number $\varepsilon$ such that $0<\varepsilon<K$ and then select $T>0$ small enough so that $\alpha(t)(x)\geq\varepsilon$ for all $t\in[0,T]$ and all $x\in[0,1].$
\end{proof}

\begin{corollary}
If $f_0$ is not a constant function, and if $f_0(x)\geq K>0$ for all $x$, then the adaptive dynamics evolves to a strategy which defeats $f_0$. 
\end{corollary}

\begin{proof}
This follows from Theorem~\ref{teo:strategy} and Lemma~\ref{lem:nonconstant}.
\end{proof}

For some initial conditions, the solution to \eqref{eq:problem} no longer can be considered a strategy, because it may assume negative values at some $x \in [0,1]$.
This typically happens if the initial data $f_0$ does not satisfy the positivity condition of Theorem~\ref{teo:strategy}.
For example, if
\begin{equation}
f_0(x)=
\begin{cases}
1-x/r, & x\in[0,r]\\
a(x-r), & x\in[r,1]
\end{cases}
\label{eq:impossible}
\end{equation}
where $\frac{1}{2}<r<1$ and $a$ is a positive number, which is selected such that $f_0$ satisfies $\mca(f_0)=\frac{1}{2}.$
Then the adaptive dynamics on $L^\infty[0,1]$ yields a solution which is negative for some $x$ for infinitesimal $t$.
The constrained selection gradient $A(f)$ at $t=0$ is shown as the dashed line in Figure~\ref{fig:impossible}.
In particular, it is negative at $x=r$.
This is interesting, because the strategies according to the adaptive dynamics \eqref{eq:problem} exist by Theorem \ref{Lp-theorem} and defeat $f_0$ by Lemma \ref{lem:nonconstant}.  However, it is not clear how one could interpret a function that assumes negative values in the function-valued team game.  It is also interesting to note that the \em only \em stationary strategies in the adaptive dynamics of the function-valued team game are the equilibrium strategies as shown in Proposition \ref{prop:stationary}.

\begin{figure}[H]
 \centering
 \includegraphics[width=0.6\columnwidth]{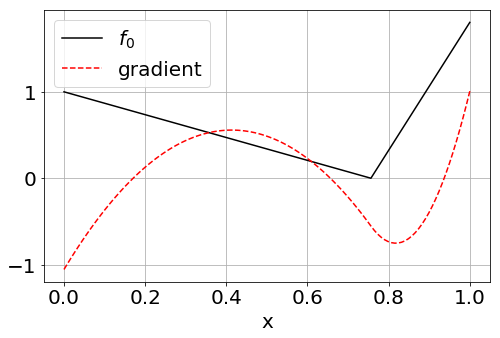}
 \caption{At the initial condition~(\ref{eq:impossible}) with $r=3/4$, the adaptive dynamics is trying to move the function away from the space of strategies by breaking the non-negativity constraint.}
 \label{fig:impossible}
\end{figure}

\subsection{Evolution towards the equilibrium}
The stationary points of the adaptive dynamics are specified in Proposition \ref{prop:stationary}.
Let $u$ be a uniform distribution over the unit interval, that is, 
$u(x)$ is a positive constant for all $x$ in the interval $0\leq x\leq 1.$
It is a stationary point for the adaptive dynamics as well as an equilibrium strategy in the game of teams.
Be ``reversing time,'' we can show that there exist strategies that will evolve to such an equilibrium strategy.
The idea is that the following problems are identical:
First, consider the forward-time problem
\begin{equation}
\dot{\alpha} = \nabla E(\alpha), \quad \alpha(0)=\alpha_0, \quad \alpha(T)=u,
\label{eq:A-goto-u}
\end{equation}
where $\alpha_0$ is a strategy such that $\mca(\alpha_0)<\frac{1}{2}.$
Then consider the reverse-time initial value problem
\begin{equation*}
\dot{\alpha} = -\nabla E(\alpha), \quad \alpha(0)=u, \quad \alpha(T)=\alpha_0.
\label{eq:A-gofrom-u}
\end{equation*}
There is no unknown in solving this initial value problem, but $\alpha_0$ is determined by assigning it $\alpha(T)=\alpha_0$. 
Compare this to~(\ref{eq:A-goto-u}), in which $\alpha_0$ is unknown.
By reversing time, we may solve an equation with unknown stopping time $T$ but with known initial data.

It remains to show that there exists a $T>0$ such that $\alpha_0=\alpha(T)$ is a strategy such that $\mca(\alpha_0)<\frac{1}{2}$.
This is however easy to prove using the following facts.
Since $u$ is a strictly positive function, 
there are functions in $L^\infty[0,1]$ that satisfy \eqref{eq:bm_cont_setup} with a strict inequality for the MCA that are arbitrarily close to $u$ in the $L^\infty$ norm. By Lemma~\ref{lem:increase-mca}, such functions have increasing MCAs.  Any such function can be used as $\alpha_0$ above to seed the initial value problem.

\section{Adaptive dynamics for the discrete team game} \label{sec:computations}

This section focuses on the discrete game of teams as described in \S \ref{sec:got}.  In the function-valued game, the measure $f(x)dx$ for an integrable function $f$ is used to define the distribution of competitive abilities. In the discrete game, the quantity of individuals with competitive ability equal to $j/M$ can be viewed as a point mass at $j/M$ weighted by the value of $f$ at $j/M$.  In this way one may compare 
\beq \label{eq:integration} 
\int_0 ^1 f(x) dx \approx  \sum_{j=0} ^M f(j/M), 
\eeq 
Then the selection gradient according to the function valued game is compared to 
\beq \label{eq:discrete-gradient} 
\nabla E f(k/M) &=& \int_0 ^{k/M} f(r) dr - \int_{k/M} ^1 f(r) dr \nn \\ 
&\approx & L \bmy = \sum_{j=0} ^{k-1} f(j/M) - \sum_{j=k+1} ^M f(j/M) . 
\eeq 
Above, the column vector $\bmy$ has components $y_j = f(j/M)$, and the $(M+1) \times (M+1)$ matrix 
\beq 
L=
\begin{bmatrix}
	0 & -1 & -1 & ... & -1\\
	1 &  0 & -1 & ... & -1\\
	1 &  1 &  0 & ... & -1\\
	\vdots &  & &     &   \\
	1 &  1 &  1 & ... &  0
\end{bmatrix}, \label{eq:defn_Lmatrix}
\eeq 
As defined in \eqref{eq:discrete-mca}, the mean competitive ability and its constraint are 
\[ \mca(\bmy) = \frac{\sum_{j=0} ^M (j/M) y_j}{\sum_{j=0} ^M y_j} \leq \frac 1 2.\]
We note that in our convention, the indices of a vector in $\R^{M+1}$ are $\bv = (v_0, v_1, \ldots, v_M)$.  We further make the assumptions on $\bmy$ given in \eqref{eq:discrete_setup}.  Then, it is straightforward to compute that the MCA constraint is an equality if and only if $\bmy$ is orthogonal to the vector with components $(j/M-1/2)$.  We therefore define the vector 
\beq \label{eq:w_discrete} \bw = \begin{bmatrix} -1/2 \\ 1/M - 1/2 \\ \vdots \\ j/M - 1/2 \\ \vdots \\ 1/2 \end{bmatrix}. \eeq 
We further compute that $\mca(\bmy) \geq 0$ is equivalent to $\bw\cdot \bmy \geq 0$. 
Consequently the projection onto the 
normal to the set of strategies $\bmy$ satisfying $\mca(\bmy)=1/2$ is given by projecting onto the span of $\bw$.  For a strategy $\bmy$ this projected vector is given by multiplying $\bmy$ on the left with the matrix
\begin{equation}
P = \frac{\bw \bw^T}{\|\bw\|^2},
\label{eq:discrete-projection}
\end{equation}
where $\|\bw\|^2=\bw^T\bw$ is the sum of the square of the components of $\bw.$
Then $I-P$, where $I$ is the $(M+1) \times (M+1)$ identity matrix, is the projection onto the set of strategies with MCA equal to $1/2$.  

Analogous to the function-valued setting, adaptive dynamics predicts that strategies evolve according to the linear ODE system
\beq \label{eq:ad_discrete0} \frac{d}{dt} \bmy = A \bmy, \quad \bmy (0) = \bmy_0, \quad A\bmy  = (I-H(\bw \cdot \bmy) P)L \bmy. \eeq 
Above $H$ is the Heaviside, so that $H(\bw\cdot\bmy)=1$ if and only if $\mca(\bmy)\geq 1/2$, otherwise it is zero.  Here $\bmy_0$ is the initial strategy that is assumed to satisfy \eqref{eq:discrete_setup} and the MCA constraint \eqref{eq:discrete-mca}.

\subsection{The evolution of strategies according to adaptive dynamics for the discrete team game}\label{sec:algebraic} 

Here we solve the ODE system \eqref{eq:ad_discrete0}. For a system of first order ODEs of this type, if an eigenvalue $\lambda$ of $A$ has multiplicity $r$ and $k$ linearly independent eigenvectors $\bv_1,...,\bv_k$ with $r=k$, 
then a basis of solutions for this eigenvalue consists of 
\[ e^{\lambda t} \bv_1, \ldots, e^{\lambda t} \bv_r.\]
If $k<r$, then the basis consists of $e^{\lambda t} \bv_j$ for $j=1, \ldots, k$ together with  $r-k$ solutions of the form $e^{\lambda t}p(t)$, where $p$ is a polynomial of degree at most $r-k$ with vector coefficients. These vector coefficients are linear combinations of generalized eigenvectors.  We recall that a generalized eigenvector for the eigenvalue $\lambda$ is a nonzero vector $\bv$ such that  for some $m \geq 1$, $(\lambda I - A)^m \bv=0$ but $(\lambda I - A)^{m-1} \bv \neq 0$.     Consequently, in order to determine the solutions of the ODE system \eqref{eq:ad_discrete0}, we first establish properties of the matrices involved in this system and determine their eigenvalues.  The projection matrix $P$ is symmetric and real, so its eigenvalues are real. Since it is a projection, all its eigenvalues are equal to either 1 or zero. The matrix $L$, which maps $\bmy$ to the unconstrained selection gradient, is an anti-symmetric Toeplitz type matrix.  It has purely imaginary eigenvalues (as do all real anti-symmetric square matrices) that occur in pairs of complex conjugates.  
If there is an odd number of eigenvalues, one of them is zero. 

\begin{lemma}
Let $L$ denote the $(M+1) \times (M+1)$ skew-symmetric matrix with all entries above the diagonal equal to $-1$, and all entries below the diagonal equal to $1$ as shown in \eqref{eq:defn_Lmatrix}.  Then the rank of $L$ is $M+1$ when $M+1$ is even, and it is $M$ when $M+1$ is odd.
\label{lem:L}
\end{lemma}

\begin{proof}
In Gauss's algorithm, we replace row $j$ with row $j$ minus row $j+1$ for all $1\leq j\leq M$ as shown below.  
\begin{equation*}
\left.
\begin{matrix}
0 & -1 & -1 & ... & -1\\
1 & 0  & -1 & ... & -1\\
1 & 1  & 0  & ... & -1\\
\vdots & & & & \vdots\\
1 & 1  & 1  & ... & 0 
\end{matrix}
\ 
\right|
\begin{matrix}
0 \\ 0 \\ 0 \\ \vdots  \\ 0
\end{matrix}
\iff
\left.
\begin{matrix}
	-1 & -1 & 0  & 0 & ... & 0\\
	0  & -1 & -1 & 0 & ... & 0\\
	0  &  0 & -1 & -1& ... & 0\\
	\vdots & & & & &\vdots\\
	1 & 1  & 1  &  1 & ... & 0 
\end{matrix}
\ 
\right|
\begin{matrix}
	0 \\ 0 \\ 0 \\ \vdots  \\ 0
\end{matrix}
\end{equation*}
If $M+1$ is even,  we add row $1$, $3$, and all of the odd rows up to row $M$ to the last row.  If $M+1$ is odd, we add row $1$, $3$, and all of the odd rows up to row $M-1$ to the last row.  In this way we obtain   

\begin{equation*}
\left.
\begin{matrix}
-1 & -1 & 0  & 0 & ... &0 & 0\\
0  & -1 & -1 & 0 & ... &0 & 0\\
0  &  0 & -1 & -1& ... &0 & 0\\
\vdots & & & & & &\vdots\\
0  &  0 & 0  & 0 & ... &-1 & -1\\
0 & 0  & 0  &  0 & ... &0 & a 
\end{matrix}
\ 
\right|
\begin{matrix}
0 \\ 0 \\ 0 \\ \vdots \\ 0 \\ 0
\end{matrix}
\quad\text{where }
a=
\begin{cases}
-1 & \text{when } M+1\text{ is even,}\\
0  & \text{when } M+1\text{ is odd.}
\end{cases}
\end{equation*}
We can deduce from this calculation that the kernel of $L$ when $M+1$ is odd is the span of the vector $(1,-1,1,-1,...,-1,1)$.
\end{proof}

\begin{proposition}
\label{prop:dimker}
Let $L$ and $P$ be defined in \eqref{eq:defn_Lmatrix} and \eqref{eq:discrete-projection}, respectively, and $I$ be the $(M+1) \times (M+1)$ identity matrix.  Then we have
\begin{equation*}
\dim \mathrm{Ker}\, (I-P)L =
\begin{cases}
2, & M+1\text{ odd}\\
1, & M+1\text{ even.}
\end{cases}
\end{equation*}
\end{proposition}

\begin{proof}
Consider the rank of $(I-P)L$.  If $M+1$ is even, then $L: \R^{M+1} \to \R^{M+1}$ is a surjection.  The matrix $I-P$ projects onto the orthogonal complement of the vector $\bw$, an $M$-dimensional subspace, so in this case the rank of $(I-P)L$ is $M$.   If $M+1$ is odd, then $L$ maps $\R^{M+1}$ to an $M$-dimensional subspace.  Each of the columns of $L$ are contained in this subspace.  Considering just the first column, it is not orthogonal to $\bw$, it follows that $L \R^{M+1}$ is not contained in the orthogonal complement of $\bw$.  Consequently, when we apply $(I-P)$ to $L \R^{M+1}$ the resulting subspace loses one dimension and is thus of dimension $M-1$. The proof then follows from the rank-nullity theorem in both cases.
\end{proof}

\begin{proposition} \label{prop:product_evals}
Let $Q$ be a projection matrix and $S$ be a real square anti-symmetric matrix of the same dimensions.  Then the eigenvalues of $QS$ are contained in $i \R$.
\end{proposition}

\begin{proof}
Since $Q$ is a projection matrix, there exists a change of basis, implemented by a unitary matrix $U$ under which $Q$ has the form 
\[ U^T Q U = \begin{bmatrix}
I_{k}  &  \bm{0} \\
\bm{0} & \bm{0} 
\end{bmatrix}
\] 
where $I_k$ is a $k\times k$ identity matrix, with $k$ equal to the rank of $Q$.  Then using this change of basis, we have 
\[ U^T (QS) U = U^T Q U U^T S U = (U^T Q U)(U ^T S U) = \begin{bmatrix}
I_{k}  &  \bm{0} \\
\bm{0} & \bm{0} 
\end{bmatrix} (U^T S U) = \begin{bmatrix}
	S'  &  S'' \\
	\bm{0} & \bm{0} 
\end{bmatrix}.
\]
where $S'$ is the upper-left $k\times k$ block of $U^T S U$ and $S''$ is the upper-right corner of $U^T S U$ of size $k\times (n-k)$.  Notice that $S'$ is also anti-symmetric, because anti-symmetric matrices are anti-symmetric with respect to any basis.
Now, if $Q$ and $S$ are $n \times n$, then we compuate that 
\begin{equation*}
\det 
\begin{bmatrix}
	\lambda I_k-S'  &  S'' \\
	\bm{0} & \lambda I_{n-k}
\end{bmatrix}
= \lambda^{n-k}\det(\lambda I_k-S')
\end{equation*}
Since $S'$ is anti-symmetric, this polynomial has roots in $i\R$, and these roots are eigenvalues of $QS$.   The remaining eigenvalue is $0$, if $n-k>0$.
\end{proof}

We now apply these results to characterize the eigenvalues of $(I-P)L$ as well as those of $L$.  

\begin{corollary} \label{cor:evals_A}
The non-zero eigenvalues of $(I-P)L$ occur in pairs of the form $\pm i b$ for nonzero $b \in \R$.  Zero is an eigenvalue of $(I-P)L$ with geometric multiplicity one if $M+1$ is even, and geometric multiplicity $2$ if $M+1$ is odd.  The non-zero eigenvalues of $L$ occur in pairs of the form $\pm i b$ for nonzero $b \in \R$.  Zero is an eigenvalue with geometric multiplicity equal to one precisely when $M+1$ is odd.  
\end{corollary}


\begin{proof} 
We note that $(I-P)L$ has all real entries, and so the characteristic polynomial $\det(\lambda I - (I-P)L)$ has real coefficients.  It therefore follows that if $z$ is a root of this polynomial, which is equivalent to being an eigenvalue of $(I-P)L$, then $\overline z$ is also a root of this polynomial.  By the preceding proposition, the eigenvalues of $(I-P)L$ are contained in $i \R$.  The non-zero ones therefore occur in pairs of the form $\pm i b$ for non-zero $b \in \R$.  The dimension of the eigenspace of the eigenvalue zero  follows from Proposition \ref{prop:dimker}.  In the same way we apply the Proposition to $L=IL$, with $I$ the identity matrix of the same dimensions as $L$, noting that this is a projection matrix, so the proposition applies in the same way. 
\end{proof}

We now determine a basis for the kernel of $(I-P)L$ in case its dimensions $(M+1) \times (M+1)$ have $M+1$ odd.  
\begin{proposition}
\label{prop:Jordanchain-2}
Assume that $(I-P)L$ is $(M+1)\times (M+1)$ for $M+1$ odd.  Then the vectors $\bv_o$ and $\bv_e$ with ones in the odd and even components, respectively, and all other entries equal to zero, constitute a basis of the kernel of $(I-P)L$. 
\end{proposition}

\begin{proof}
Let $\bv$ be either of the solutions in the statement of the proposition.
We compute 
\[
L \bv = -\frac M 2 
\begin{bmatrix}
1 \\ 1 \\ \vdots \\ 1 \\ 1
\end{bmatrix}
+
\begin{bmatrix}
0 \\ 1 \\ 2 \\ \vdots \\ M 
\end{bmatrix}.
\]
Since the $\mca$ of the first vector is equal to $1/2$, it is orthogonal to $\bw$, and so left multiplication with $P$ yields the zero vector.  Let $\bm{u}$ be the second vector.  

Then 
\[
P\bm{u} = 
\begin{bmatrix}
0 \\ 1 \\ 2 \\ \vdots \\ M 
\end{bmatrix}
- \frac M 2 
\begin{bmatrix}
1 \\ 1 \\ \vdots \\ 1 \\ 1
\end{bmatrix}
\]
The conclusion is $(I-P)L\bv =L\bv - PL\bv=(0,0,...,0)$.
Since these $\bv_o$ and $\bv_e$ are linearly independent and since $\dim\mathrm{Ker}\, (I-P)L=2$ by Proposition~\ref{prop:dimker}, they constitute a basis.
\end{proof}

We now determine a basis for the kernel of $(I-P)L$ in case its dimensions $(M+1) \times (M+1)$ have $M+1$ even.  
\begin{proposition}
\label{prop:Jordanchain-1}
Assume that $M+1$ is even.  Define the vector $\bv_2$ $=$ \\ $(2,0,2,0,...,2,0)$.
Then $\bv_2$ solves $(I-P)L\bv_2 = (1,1,1,...,1)$ $=\bv_1$, and $\bv_1$ is a basis for the kernel of $(I-P)L$.   That is, $\bv_2$ and $\bv_1$ consitute a Jordan chain for the eigenvalue zero. 
\end{proposition}

\begin{proof}
Let $\bv_2$ be as in the statement of the proposition and compute 
\[
L \bv_2 = -(M-1) 
\begin{bmatrix}
1 \\ 1 \\ \vdots \\ 1 \\ 1
\end{bmatrix}
+
2
\begin{bmatrix}
0 \\ 1 \\ 2 \\ \vdots \\ M
\end{bmatrix} = 2M \bw + \bv_1.
\]
Then, since $\mca(\bv_1) = 1/2$ it is in the orthogonal complement of the span of $\bw$.  Consequently, since $I-P$ projects onto the orthogonal complement of $\bw$, $(I-P)L\bv_2= \bv_1$. We then also compute that 
\[ L \bv_1 = -2M \bw \implies (I-P)L\bv_1 = 0.\]
\end{proof}

The fundamental theorem for linear systems \cite{perko2001} gives the explicit form of the solution to \eqref{eq:ad_discrete0}.  We first give the case when the initial data satisfies the MCA constraint with equality.

\begin{theorem} \label{th:discrete_sol_mca_equal} Assume that the initial data $\bmy_0 \in \R^{M+1}$ satisfies \eqref{eq:discrete_setup} and the MCA constraint \eqref{eq:discrete-mca} is an equality.  Then the solution to the adaptive dynamics \eqref{eq:ad_discrete0} with initial data $\bmy_0$ is 
\[ \bmy(t) = Q B Q^{-1} \bmy_0.\]
The matrix $B$ is a real block-diagonal square matrix, and $Q$ is real and invertible. These matrices both have dimensions $(M+1) \times (M+1)$.  The columns of $Q$ are the generalized eigenvectors of $(I-P)L$ ordered in the following way:
Let $m$ be the algebraic multiplicity of the eigenvalue $\lambda=0$, and let $\bm{v}_1,\bv_2,...,\bv_m$ be a set of generalized eigenvectors for $\lambda=0$. 
Non-zero eigenvalues occur in conjugate pairs $\lambda=i\beta_j$, $\bar{\lambda}=-i\beta_j$, for $j=m+1,m+2,...,\ell$. 
If $\bm{u}_j+i\bm{w}_j$ is a generalized eigenvector to $\lambda=i\beta_j$, where $\bm{u}_j$ and $\bm{w}_j$ are the real and imaginary part of the generalized eigenvector, then a basis for $\R^{M+1}$ and the columns of $Q$ are given by
\[
\bv_1,\bv_2,...,\bv_m, \bm{u}_{m+1},\bm{w}_{m+1},...,\bm{u}_\ell,\bm{w}_\ell.
\]
The first generalized eigenvectors are given by Proposition~\ref{prop:Jordanchain-2} if $M+1$ is odd or Proposition~\ref{prop:Jordanchain-1} if $M+1$ is even.
Correspondingly, if $M+1$ is odd then $m\geq 3$ and is odd, and if $M+1$ is even then $m\geq 2$ and is even.
The matrix $B$ consists of blocks along the diagonal corresponding to the eigenvalues, ordered like the above basis.  Each block $B_j$ corresponding to the zero eigenvalue, $0$, is of the form 
\begin{equation*}
\begin{bmatrix} 
1 & t & \frac{t^2}{2} & \dots & \frac{t^{k-1}}{(k-1)!} \\
 & 1 & t & \dots & \\ 
 &   & 1 & \ddots & \\
 &   &   & \ddots & t\\
 & & & & 1 
\end{bmatrix}
\end{equation*}
for some $k \geq 1$, and such that the sizes of these blocks add up to the algebraic multiplicity of $\lambda=0$. All components below the diagonal are zero. For an eigenvalue $i \beta_j \neq 0$, let 
\[ R_j = \begin{bmatrix} \cos(\beta_j t) & -\sin(\beta_j t) \\ \sin(\beta_j t) & \cos(\beta_j t) \end{bmatrix}. \]
A block $B_j$ corresponding to $\pm i \beta_j$ is of the form
\[ \begin{bmatrix} R_j & tR_j & \frac{t^2}{2}R_j & \ldots & \frac{t^{k-1}}{(k-1)!}R_j  \\ & R_j & tR_j & \ldots &  \\ & & \ddots & \ddots & \\ & & & R_j & tR_j \\ & & & & R_j 
\end{bmatrix} \]
for some $k \geq 1$.  All components below the diagonal blocks are zero. 
\end{theorem} 

\begin{proof} 
The form of the solution follows immediately from the fundamental theorem for linear systems and the real Jordan form \cite{perko2001} of the matrix $(I-P)L$, noting that Propositions~\ref{prop:dimker}--\ref{prop:Jordanchain-2} apply to this matrix.  We note that the algebraic multiplicity is greater than or equal to the geometric multiplicity.  Thus when $M+1$ is even, since the blocks corresponding to nonzero eigenvalues are all even-dimensional, the block corresponding to the zero eigenvalue must also be even-dimensional.  Since the geometric multiplicity of the zero eigenvalue is one, this shows that its algebraic multiplicity is at least two.  When $M+1$ is odd, the geometric multiplicity of the zero eigenvalue is two.  Since the blocks corresponding to the non-zero eigenvalues are all even-dimensional, the block corresponding to the zero eigenvalue must be odd dimensional.  Therefore the algebraic multiplicity of the zero eigenvalue is at least 3.  
\end{proof}

The same arguments, together with Lemma \ref{lem:L} and Corollary \ref{cor:evals_A} gives the solution in case the initial data has MCA strictly less than $1/2$.  

\begin{theorem} \label{th:discrete_solution_mca_less}
Assume that the initial data $\bmy_0 \in \R^{M+1}$ satisfies \eqref{eq:discrete_setup} and the MCA constraint \eqref{eq:discrete-mca} is a strict inequality.  Then the solution to the adaptive dynamics \eqref{eq:ad_discrete0} with initial data $\bmy_0$ is 
\[ \bmy(t) = Q B Q^{-1} \bmy_0.\]
If $M+1$ is even, then  the $(M+1) \times (M+1)$ invertible matrix $Q$ has columns given by the real and imaginary parts of the eigenvectors and generalized eigenvectors of $L$.  The matrix $B$ consists of blocks along the diagonal corresponding to the eigenvalues of $L$.
These blocks are of the same type as in Theorem \ref{th:discrete_sol_mca_equal}.
If $M+1$ is odd, then one block is $1$, that is a $1\times 1$ block corresponding to the zero eigenvalue, which has algebraic and geometric multiplicity equal to one. The remaining blocks correspond to the eigenvalues $\pm i\beta_j$ for real $\beta_j\neq 0$. 
\end{theorem} 

\begin{proof}
We compute the characteristic polynomial of $L$ using induction on the size of $L$.
Let $L$ be as in~(\ref{eq:defn_Lmatrix}) of size $(n+1)\times (n+1)$ for some positive integer $n$ and let $L_n$ be the same matrix but of size $n\times n$.  We begin by subtracting the $(j+1)^{st}$ row from the $j^{th}$ row starting from the first row and continuing to the last row, keeping the last row unchanged.  Then we calculate the determinant by expanding along the first column obtaining
\beq 
&\det(L-\lambda I)=
\det 
\begin{bmatrix}
-\lambda & -1 & -1 & ... & -1\\
1 & -\lambda  & -1 & ... & -1\\
1 & 1  & -\lambda & ... & -1\\
\vdots & & & \ddots & \vdots\\
1 & 1  & 1  & ... & -\lambda 
\end{bmatrix} \nn
\eeq
\beq
&= 
\det
\begin{bmatrix}
-1-\lambda & \lambda-1 & 0 & ... & 0\\
0 & -1-\lambda  & \lambda-1 & \ddots &\vdots \\
0 & 0  & \ddots  & \ddots & 0\\
\vdots & & & -1-\lambda & \lambda-1\\
1 & 1  & \dots & 1 & -\lambda 
\end{bmatrix} \nn \\
&=(-1-\lambda) \det (L_{n}-\lambda I_{n})
+ (-1)^{n+2}\det
\begin{bmatrix}
 \lambda-1 & 0 & ... & 0\\
 -1-\lambda  & \lambda-1 & ... & 0\\
 0  & \ddots  & \ddots & \vdots\\
 \vdots & & -1-\lambda & \lambda-1
\end{bmatrix} \nn \\
&=(-1-\lambda) \det (L_{n}-\lambda I_{n})
+ (-1)^{n+2}(-1+\lambda)^n.  \label{eq:clever_det}
\eeq 

We claim that 
\beq \det(L-\lambda I) = \sum_{k=0} ^{\frac{n+1}{2}} {n+1 \choose 2k} \lambda^{2k}, \textrm{ if $n+1$ is even}, \label{eq:ind_n+1_even} \eeq 
and 
\beq \det(L-\lambda I) = - \lambda \sum_{k=0} ^{ \frac n 2} {n+1 \choose 2k+1} \lambda^{2k}, \textrm{ if $n+1$ is odd.} \label{eq:ind_n+1_odd} \eeq 
Once these expressions are established, it is immediately apparent that $0$ is not an eigenvalue of $\det(L-\lambda I)$ when $n+1$ is even, and it is an eigenvalue of algebraic multiplicity one when $n+1$ is odd. It is also apparent that all other eigenvalues are purely imaginary and occur in conjugate pairs.  So, to complete the proof, we demonstrate \eqref{eq:ind_n+1_even} and \eqref{eq:ind_n+1_odd}.

We calculate directly that 
\begin{equation*}
\det
\begin{bmatrix}
-\lambda & -1\\
1 & -\lambda
\end{bmatrix}
=\lambda^2 +1,\quad
\det
\begin{bmatrix}
-\lambda & -1 & -1\\
1 & -\lambda & -1\\
1 & 1 & -\lambda
\end{bmatrix}
=-\lambda^3 -3\lambda.
\end{equation*}
This demonstrates the base cases.  Using \eqref{eq:clever_det} and the induction assumption for $n$ odd, we compute $\det(L-\lambda I)$ in dimension $n+1 \times n+1$ is 
\[(1+\lambda) \sum_{k=0} ^{\frac{n-1}{2}} {n \choose 2k+1} \lambda^{2k+1} + (-1)^n \sum_{j=0} ^n {n \choose j} (-1)^{n-j} \lambda^j\] 
\[ = \sum_{k=0} ^{\frac{n-1}{2}} {n \choose 2k+1} \lambda^{2k+1} + \sum_{k=0} ^{\frac{n-1}{2}} {n \choose 2k+1} \lambda^{2k+2} + \sum_{j=0} ^n {n \choose j} (-1)^{j} \lambda^j.\]
We split the sum 
\[\sum_{j=0} ^n {n \choose j} (-1)^{j} \lambda^j = - \sum_{k=0} ^{\frac{n-1}{2}} {n \choose 2k+1} \lambda^{2k+1} + \sum_{k=0} ^{\frac{n-1}{2}} {n \choose 2k} \lambda^{2k}. \]
This shows that the determinant simplifies to 
\[ \sum_{k=0} ^{\frac{n-1}{2}} {n \choose 2k+1} \lambda^{2k+2} + \sum_{k=0} ^{\frac{n-1}{2}} {n \choose 2k} \lambda^{2k}. \]
We re-index the first sum by setting $j=k+1$ and obtain 
\[ \sum_{j=1} ^{\frac{n+1}{2}} {n \choose 2j-1} \lambda^{2j} + \sum_{k=0} ^{\frac{n-1}{2}} {n \choose 2k} \lambda^{2k} = 1 + \lambda^{n+1} + \sum_{k=0} ^{\frac{n-1}{2}} \left[ {n \choose 2k} + {n\choose 2k-1} \right] \lambda^{2k}\]
\[ = 1 + \lambda^{n+1} + \sum_{k=0} ^{\frac{n-1}{2}} {n+1 \choose 2k} \lambda ^{2k} = \sum_{k=0} ^{\frac{n+1}{2}} {n+1 \choose 2k} \lambda^{2k}.\] 
Since $n$ is odd, $n+1$ is even, and this is indeed \eqref{eq:ind_n+1_even}. 

Next, assume our claim holds for $n$ even, and we calculate using \eqref{eq:clever_det} and the induction assumption: 
\[ \det(L-\lambda I) = -(1+\lambda) \sum_{k=0} ^{\frac n 2} {n \choose 2k} \lambda^{2k} + (-1)^n \sum_{j=0} ^n {n \choose j} (-1)^{n-j} \lambda^j \] 
\[ = - \sum_{k=0} ^{\frac n 2} {n \choose 2 k} \lambda^{2k} - \lambda \sum_{k=0} ^{\frac n 2} {n \choose 2 k} \lambda^{2k} + \sum_{j=0} ^n {n \choose j} (-1)^j \lambda^j.\]
Splitting the sum in $j$, the first summand cancels resulting in the simplification to 
\[ - \lambda \sum_{k=0} ^{\frac n 2} {n \choose 2 k} \lambda^{2k} - \sum_{j=0} ^{n/2-1} {n \choose 2j+1} \lambda^{2j+1} \]
\[ = - \lambda - \lambda^{n+1} - \lambda \sum_{k=1} ^{n/2-1} \left[{n \choose 2k} + {n \choose 2k+1} \right] \lambda^{2k} \]
\[ = - \lambda - \lambda^{n+1} - \lambda \sum_{k=1} ^{n/2-1} {n+1 \choose 2k+1} \lambda^{2k} = - \lambda \sum_{k=0} ^{n/2} {n+1 \choose 2k+1} \lambda^{2k}.\]
This is indeed \eqref{eq:ind_n+1_odd}. 
\end{proof}

It follows from the preceding two theorems that the solution to \eqref{eq:ad_discrete0} is of the form 
\[ 
\bc_0(t) + \sum \ba_k(t) \cos(\beta_k t) + \bb_k (t) \sin(\beta_k t).
\]
Here, $\bc_0(t)$, $\ba_k(t)$, and $\bb_k(t)$ are polynomials of the variable $t$ with vector-valued coefficients, which are linear combinations of the generalized eigenvectors of the eigenvalues 0 and $\pm i\beta_k$, respectively. In case the MCA constraint is satisfied with a strict inequality, and $M+1$ is even, then $\bc_0$ vanishes.
In the case of Theorem \ref{th:discrete_solution_mca_less} when the initial data satisfies the MCA constraint with a strict inequality, neither the MCA nor the sum of the components of the solution are constant.  
We prove this in Proposition \ref{prop:growth_mca_less_12} and also observe it in numerical experiments as shown in Figure \ref{fig:growing}. 

\begin{proposition} \label{prop:growth_mca_less_12}
Let $\bmy$ be a solution to $\dot{\bmy}=L\bmy$ with initial data $\bmy_0$ satisfying \eqref{eq:discrete_setup} and $\mca(\bmy_0)<\frac{1}{2}$.  Then $\mca(\bmy)$ and $\sum_j y_j$ are increasing functions of time at all times for which the solution satisfies $\mca(\bmy)<1/2$, $y_j\geq 0$ for all $j$, and $\sum_j y_j>0$.
\end{proposition} 

\begin{proof}
Notice that 
\begin{equation}
\label{eq:sum-Ly}
\mca(\bmy) < \frac 1 2 \iff 0 > \sum_{j=0}^M (2j-M) y_{j} = -\sum_{j=0}^{M} (L\bmy)_j = - \frac{d}{dt} \sum_{j=0} ^M y_j.
\end{equation}
Therefore, it follows from $\mca(\bmy)<\frac{1}{2}$ that $\sum y_j$ is increasing. Next, we will show that the MCA increases as $t$ increases.

Since $\mca(\bmy)<\frac{1}{2}$,
\[
\sum_{j=0}^M (j/M) y_j <\frac{1}{2}\sum_{j=0}^M y_j.
\]
Then using this equation together with $\dot \bmy = L \bmy$ and the definition of the MCA, we compute 
\begin{align*}
\frac{d}{dt}\mca(\bmy(t)) 
 &= \frac{\left(\sum_i (i/M) (L\bmy)_i\right)\left(\sum_j y_j\right)-\left(\sum_i (i/M)y_i\right)\left(\sum_j (L\bmy)_j\right)}{\left(\sum_k y_k\right)^2}\\
 &>\frac{\sum_i (i/M-\frac{1}{2})(L\bmy)_i}{\sum_k y_k}
\end{align*}
It remains to prove that $\sum_i (i/M-\frac{1}{2})(L\bmy)_i\geq 0$.
By the definition of $L$,
\begin{align*}
\sum_i (i/M)(L\bmy)_i 
=& \phantom{+}\,\frac{0}{M}(-y_1-y_2-...-y_M)\\
 & +\frac{1}{M}(y_0-y_2-...-y_M)\\
 & +\frac{2}{M}(y_0+y_1-y_3-...-y_M)\\
 & +...\\
 & +\frac{M}{M}(y_0+y_2+...+y_{M-1})\\
=&  \sum_{i=0}^M y_i \left(\frac{M(M+1)}{2M}-\frac{(i+1)i}{2M}-\frac{i(i-1)}{2M}\right)\\
=&  \sum_{i=0}^M y_i \left(\frac{M(M+1)}{2M}-\frac{2i^2}{2M}\right).
\end{align*}
Using equation~(\ref{eq:sum-Ly}),
\begin{align*}
\sum_i \left((i/M)-\frac{1}{2}\right)(L\bmy)_i 
=& \sum_{i=0}^M y_i\left(\frac{M(M+1)}{2M}-\frac{2i^2}{2M}-\frac{M}{2}+i\right)\\
=& \sum_{i=0}^M y_i\left(\frac{1}{2}-\frac{i^2}{M}+i\right)\\
=& \sum_{i=0}^M y_i\left(\frac{1}{2}+\frac{(M-i)i}{M}\right).
\end{align*}
We conclude that
\[
\frac{d}{dt}\mca(\bmy(t)) > \frac{1}{\sum_k y_k}\left(\sum_{i=0}^M y_i\left(\frac{1}{2}+\frac{(M-i)i}{M}\right)\right)>0.
\]
The last inequality follows from the assumption that $y_k \geq 0$ for all $k$ and $\sum_k y_k > 0$.
\end{proof}

A numerical example of growing population size, $\sum_i y_i$, is shown in figure~\ref{fig:growing}. 
In this example, the initial data has a mean competitive ability, MCA, strictly less than $\frac{1}{2}$. 
Therefore, $\frac{d}{dt}\sum_i y_i>0$ until $\mca(\bmy)=\frac{1}{2}$, which is compatible with the results of proposition~\ref{prop:growth_mca_less_12}.  Theorems \ref{th:discrete_sol_mca_equal} and \ref{th:discrete_solution_mca_less} and Proposition \ref{prop:growth_mca_less_12} enable us to identify all stationary solutions of the adaptive dynamics.  Interestingly, these are precisely the equilibrium strategies of the discrete team game.

\begin{figure}[h]
 \centering
 \includegraphics[width=0.5\textwidth]{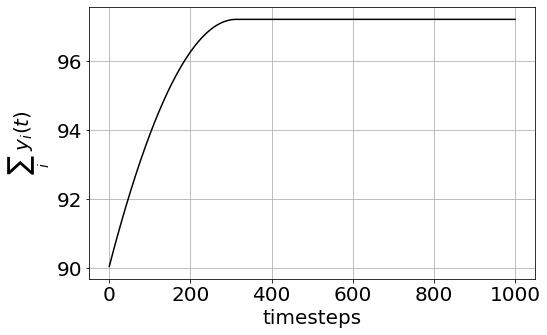}
 \caption{\label{fig:growing}In this numerical example, $\mca(\bmy_0)<\frac{1}{2}$ and the sum of the components, $\sum_i y_i(t),$ is growing until $\mca(\bmy)=\frac{1}{2}$.}
\end{figure}

\begin{corollary} \label{cor:discrete_stationary} 
The stationary solutions of the discrete team game adaptive dynamics \eqref{eq:ad_discrete0} with initial data satisfying \eqref{eq:discrete_setup} and the MCA constraint \eqref{eq:discrete-mca} are precisely the equilibrium strategies of the discrete team game.  
\end{corollary} 

\begin{proof} 
A stationary solution must satisfy $ \dot \bmy = 0, \quad \bmy(0) = \bmy_0$. 
Thus, it will be equal to the initial data for all time.  It therefore follows from Proposition \ref{prop:growth_mca_less_12} that for initial data with $\mca(\bmy_0) < 1/2$, it cannot be a stationary solution.  We are therefore left with the case when the initial data satisfies the MCA constraint with equality.  Since a polynomial cannot be identically equal to a nonconstant trigonometric function, the polynomial term $\bc_0(t)$ in the solution must be constant.  Similarly, all of the trigonometric terms must cancel in order to remain constant.  The polynomial term arises from the 0 eigenvalue of $(I-P)L$ and its eigenvectors together with its generalized eigenvectors.  The constant term in $\bc_0(t)$ is a linear combination of the eigenvectors, whereas any non-constant terms in $\bc_0(t)$ arise from the generalized eigenvectors.  The eigenvectors of $(I-P)L$ for the eigenvalue $0$ are given in Propositions \ref{prop:Jordanchain-2} and \ref{prop:Jordanchain-1}.  We note that the span of these eigenvectors consists precisely of the equilibrium strategies of the discrete team game.  
\end{proof}

\subsection{Evolution to an equilibrium strategy} \label{ss:discrete_evo_eq}
The equilibrium strategies for the discrete game are precisely the stationary points of the adaptive dynamics.  Consequently, they remain unchanged by the evolution according to the adaptive dynamics. It may not be immediately apparent that there exist strategies that will evolve to an equilibrium strategy for the discrete game.  To demonstrate the existence of such strategies, consider the problem of finding $\bmy_0$ such that 
\begin{equation}
\dot \bmy = L \bmy, \quad \bmy(0)=\bmy_0, \quad \bmy(T)=(1,1,\ldots, 1).
\label{eq:goto-u}
\end{equation}
By reversing the time variable, this is equivalent to solving the initial value problem: 
\[ \dot \bmy = - L \bmy, \quad \bmy(0) = (1, 1, \ldots, 1), \quad \bmy(T) = \bmy_0.\]
In solving this initial value problem, there is no unknown. Instead, we determine $\bmy_0$ by assigning it $\bmy(T)=\bmy_0$. 
Whenever $\bmy(0)$ has strictly positive elements there is always a $T>0$ such that the problem is solvable, and such that $\bmy(T)$ also has positive elements. 
Since $T$ can be chosen freely, as long as $\bmy(T)$ has non-negative elements, there is a one-parameter family of initial values leading to the stationary solution.
The requirement that $\bmy(T)$ has non-negative elements typically implies that $T$ cannot be very large, but maybe more importantly, $T$ can be arbitrarily small.
Figure~\ref{fig:y-to-u} shows a numerical example of this.  Although the equilibrium strategies of the game are stationary  points for the adaptive dynamics, they are not stable.  In \S \ref{sec:branching} we show that any perturbation of the stationary solution unsettles the system. 

\begin{figure}
 \centering
 \includegraphics[width=0.5\textwidth]{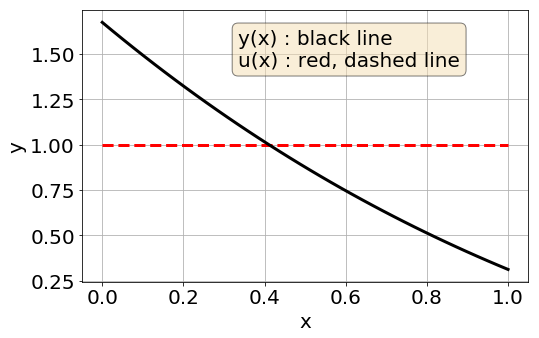}
 \caption{\label{fig:y-to-u}In this numerical example, the function $y_0(x)$ is plotted against the constant function $u(x) \equiv 1$. For this particular example,  $\mca(\bmy_0)<\frac{1}{2}$ and the dynamics is described by equation~(\ref{eq:goto-u}).}
\end{figure}

Given the linearity of the evolution equation $\dot \bmy=A\bmy$ it is tempting to imagine that a species which initially has $y_j>0$ for all $j$ can be ``split'' into two species, one of which is $\bm{a} = (a, a, a, \ldots, a)$ for $a=\min\{y_j\}$.  This $\bm{a}$ species is constant under the adaptive dynamics.  Assume that $\mca(\bmy_0) = 1/2$.
Then we have 
$\bmy(t)= \bm{a} + \bv(t)$, where $v_j\geq 0$, for all $j=0,1,...,M$.
Since $\mca(\bm{a}) = 1/2$, and $\mca(\bmy_0) = 1/2$, the initial data $\bv_0$ also satisfies $\mca(\bv_0)=\frac{1}{2}$.
However, if $\bv_0$ is not a stationary solution then $\bmy$ evolves, and at a later time $t>0$ it could be better than $\bmy_0$ in the sense that 
$E[\bmy(t),\bmy_0]>0$.    It could also happen that $\bv$ evolves in such a way that $\sum v_j$ decreases, causing the population $\sum y_j = \sum a+v_j$ to shrink.  So, although the $\bm{a}$ subspecies has constant population, the other part of the species does not have this guarantee.  In fact, it could occur that some $v_j$ become negative.  So, there is no way to see that a stable subspecies can safeguard  ``the whole of the species'' from neither mutation nor from attrition.  

\subsection{Branching}\label{sec:branching}
We explore the possibility of \emph{branching} by perturbing the equilibrium strategy $(1, 1, \ldots, 1)$ (with 51 sample points $k=0,1,...,50$, so $M=50$) by a small amount at its midpoint.
The perturbed strategy $\bmy$ has $y_k=1$ for all $k$ except $k=25$, with either $y_{25}=0.99$ or $y_{25}=1.01$.  
Figure \ref{fig:up1} shows the perturbed strategy with $y_{25}=1.01$ as a red line and the resulting branch as the black line, whereas figure \ref{fig:down1} shows the the perturbed strategy with $y_{25}=0.99$ as a red line and the corresponding branching strategy in black.
The evolved strategies in the two cases are very similar.
They are mirror images of each other when mirrored through the equilibrium strategy $(1, 1, \ldots, 1)$.  

\noindent
\begin{minipage}[t]{0.47\textwidth}
	\begin{figure}[H]
		\includegraphics[width=\textwidth]{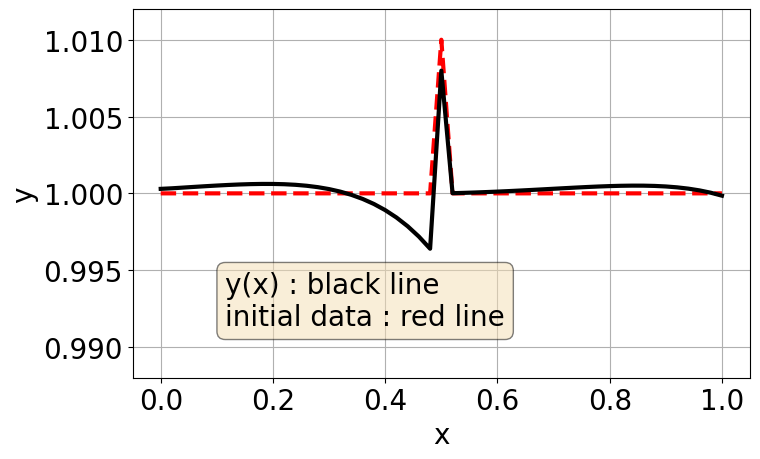}
		\caption{Perturbation of $(1, 1, \ldots, 1)$ by $+0.01$ at its midpoint.}
		\label{fig:up1}
	\end{figure}
\end{minipage}%
\hfill
\begin{minipage}[t]{0.47\textwidth}
	\begin{figure}[H]
		\centering
		\includegraphics[width=\textwidth]{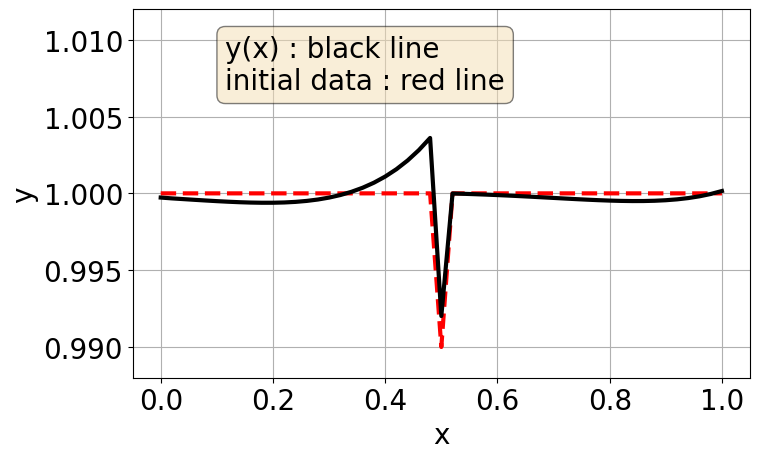}
		\caption{Perturbation of $(1, 1, \ldots, 1)$ by $-0.01$ at its midpoint.}
		\label{fig:down1}
	\end{figure}
\end{minipage}\\[0.8em]

The results in Figures~\ref{fig:up1} and~\ref{fig:down1} can be understood from the linearity of the system \eqref{eq:ad_discrete0} and the explicit form of the solution depending on the initial data given in Theorems \ref{th:discrete_sol_mca_equal} and \ref{th:discrete_solution_mca_less}.  Thus a small perturbation $\pm 0.01 \bm{e}_{26}$, of the initial data $(1, 1, \ldots, 1)$ with $\bm{e}_{26}$ the $26^{th}$ standard unit vector in $\R^{51}$ drives the evolution in opposite directions for the two opposite signs.

\section{Discussion and outlook}\label{sec:discussion}
A key feature of our team games is the linearity of the payoff functions in their definitions. Linear payoff functions are not applicable in the classification theory that Geritz \textit{et al.}~\cite{geritz1996evolutionarily} established for adaptive dynamics. Indeed, a convergence-stable stationary point in the adaptive dynamics evolution is such that (i) its second derivative with respect to the mutant's strategy is positive and (ii) its second derivative with respect to the resident population's strategy is larger than the second derivative with respect to the mutant's strategy.  If the payoff is a linear function of the mutant's strategy, then according to Geritz \textit{et al.}~\cite{geritz1996evolutionarily} ``once the singular strategy has been established, all mutations are neutral.''  Even though this conclusion is reasonable, our results show that the absence of dynamics is a unique feature of the equilibrium strategies of the game.  
A stationary point cannot be attractive in a linear game, but considering that mutations are random in theory it can be argued that branching is possible in the team game. 

We can make this argument theoretically. 
Since the adaptive dynamics setting is a deterministic approximation to a mutation process, which is random, the underlying model assumes that the traits of a species are developing randomly.
Thus, the strategies in the current work can be thought of as approximations to traits that are in fact less predictable.  From this point of view, we can expect that unstable or neutral stationary points in linear adaptive dynamics are idealizations, and then it would be reasonable to ask what happens if the stationary solutions are perturbed.  In \S \ref{sec:branching} we showed examples of this for the discrete team game.

%
%

\subsection{Comparison between the function-valued and discrete games} \label{sec:comparison}
In both the function-valued and the discrete team games, it is not clear how to interpret strategies that assume negative values.  For this reason, we assume that the initial data is non-negative.  For certain initial conditions, the adaptive dynamics may immediately result in either a function $f$ that assumes negative values or a vector $\bmy$ that has some component $y_j < 0$.  In particular, this can occur if the initial data $f_0$ vanishes at some points in $[0,1]$ or the initial data $\bmy_0$ has some components $y_j = 0$.

To compare this phenomenon for the discrete and function valued games, we consider samples of $f_0(x)=(x-\frac{1}{2})^2$ at points $x_j=j/M$ with $0\leq j\leq M$ for some integer $M>2$.
Then $f_0$ is positive at the $x_j$ which is closest to (but not equal to) $\frac{1}{2}$.
If $M=6$, then $f_0(x_3)=f_0(1/2)=0$, and $f_0(x_4)=1/36.$
Thus, the Newton forward integration would work for small step sizes, since
\[
f_0(x_4)+\varepsilon(1-P)\nabla E(f_0)(x_4)= \left(\frac{1}{6}\right)^2+\varepsilon\left(\frac{2}{3}\left(\frac{1}{6}\right)^3-\frac{1}{10}\frac{1}{6}\right)=\frac{1}{36}-\frac{11\varepsilon}{810}.
\]
Then, the $f_0(x_4)+\varepsilon(1-P)\nabla E(f_0)(x_4)>0$ for small $\varepsilon$.
This is visualized in figure~\ref{fig:non-positive}, where the red, dashed line is $f_0+\varepsilon(1-P)\nabla E(f_0)$ with $\varepsilon=2$.
For $x$ larger than $\frac{1}{2}$ but still sufficiently close to $\frac{1}{2}$, the selection gradient changes the strategy to negative values, shown by the red dashed line in figure~\ref{fig:non-positive}.
The red dots in the same figure show the samples of the evolved strategy at $x_0,x_1,...,x_6$, and none of them are negative.

It should be noted that whereas the initial condition $f_0(x)=(x-\frac{1}{2})^2$ remains a  strategy for short times in the discrete game, it may eventually also evolve to have some negative components.  We also expcet that there are strategies that for both the discrete and function valued game immediately evolve towards negative values and thus do not represent strategies in the games.  A good example of such ``impossible'' strategies is $f_0$ in~(\ref{eq:impossible}).

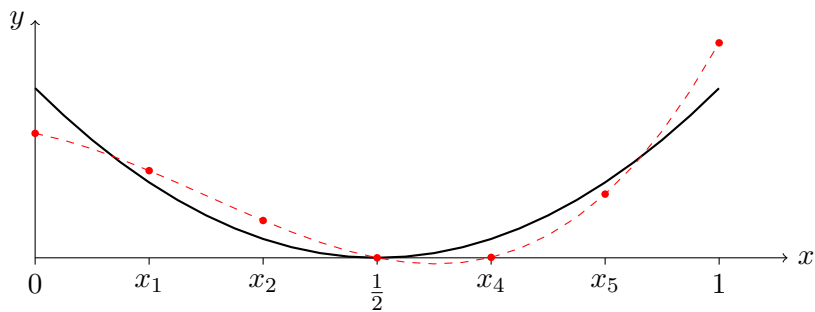
\begin{figure}[H]
\centering
\begin{tikzpicture}[scale=9,domain=0.3:0.7]
  \draw[->] (0,0) -- (1.1,0) node[right] {$x$};
  \draw[->] (0,-0.01) node[below] {$0$} -- (0,0.35) node[left] {$y$};
  \draw     (0.5,0)  -- (0.5,-0.01) node[below] {$\frac{1}{2}$};
  \draw     (5/6,0)  -- (5/6,-0.01) node[below] {$x_5$};
  \draw     (4/6,0)  -- (4/6,-0.01) node[below] {$x_4$};
  \draw     (2/6,0)  -- (2/6,-0.01) node[below] {$x_2$};
  \draw     (1/6,0)  -- (1/6,-0.01) node[below] {$x_1$};
  \draw     (1,0)  -- (1,-0.01) node[below] {$1$};
  \draw[color=black, thick, domain=0.0:1.0] plot (\x, { (\x-0.5)*(\x-0.5)});
  \draw[color=red,dashed,domain=0.0:1.0] plot (\x, { (\x-0.5)*(\x-0.5)+2*(0.6667*(\x-0.5)*(\x-0.5)*(\x-0.5)-0.1*(\x-0.5))});
  \node at (1.0000,0.31666667) [circle,fill,color=red,inner sep=1pt] {};
  \node at (0.8333,0.09382716) [circle,fill,color=red,inner sep=1pt] {};
  \node at (0.6667,0.00061728) [circle,fill,color=red,inner sep=1pt] {};
  \node at (0.5000,0.0000) [circle,fill,color=red,inner sep=1pt] {};
  \node at (0.3333,0.05493827) [circle,fill,color=red,inner sep=1pt] {};
  \node at (0.1667,0.12839506) [circle,fill,color=red,inner sep=1pt] {};
  \node at (0.0000,0.18333333) [circle,fill,color=red,inner sep=1pt] {};
\end{tikzpicture}
\caption{The samples of a strategy may be positive even if the underlying function would not remain non-negative during the evolution.}
\label{fig:non-positive}
\end{figure}

\subsection{Treating the constraint at all times}
If the MCA constraint is an equality, then in both games the projection $I-P$ is applied, whereas when it is a strict inequality, then we do not project.  This results in a discontinuity in the formulation of the dynamical system.  
To see how this affects the solution of the adaptive dynamics in the discrete case, Figure~\ref{fig:stepfunc} shows an initial data with low mean competitive ability and the corresponding solution after 500 timesteps. 
The solution's MCA grows and at $t\approx 300$ timesteps, there is a sharp change of direction of the components, as seen in figure~\ref{fig:y2}. 
This is a result of the dynamics changing equation from $\dot \bmy = L \bmy$ to $\dot \bmy = (I-P)L \bmy$. 
The $\mca$ transitions abruptly to the constant $\mca(\bmy)=\frac{1}{2}$, and as can be seen in Figures~\ref{fig:mca-evolution-1} and~\ref{fig:mca-evolution-2}, this is not an artefact of the integration step size. 
In Figure~\ref{fig:mca-evolution-1}, the stepsize $\varepsilon$ in the integration method is 0.02 whereas it is 0.0002 in Figure~\ref{fig:mca-evolution-2}. 
Consequently, the point where $\mca(\bmy)=\frac{1}{2}$ is reached after 30 iterations and 3000 iterations, respectively.
This result is independent of the sampling frequency $h=\frac{1}{M}$ in~(\ref{eq:integration}). 

The $\mca$ of a function-valued strategy is an increasing function of time, whenever $\mca(f_0)<\frac{1}{2}$, according to Lemma~\ref{lem:increase-mca}.
The lower bound is expressed in the inequality 
$\mca(\alpha_t)>\frac{1}{2}-\frac{1}{c_0+2t}$, $t>0$,
where $c_0$ is determined by the initial value $\mca(f_0)=\frac{1}{2}-\frac{1}{c_0}.$
How close to being an equality is this?
The numerical results in Figures~\ref{fig:mca-evolution-1} and~\ref{fig:mca-evolution-2} reaches $\mca(\alpha_t)=\frac{1}{2}$ quickly even if the MCA of the initial condition is strictly less than $\frac{1}{2}$.
The lower bound $g(t)=\frac{1}{2}-\frac{1}{c_0+2t}$, however, does not reach $\mca(g)=\frac{1}{2}$ at finite time, so the real MCA is significantly larger than this lower bound.

\noindent
\begin{minipage}[b]{0.47\textwidth}
	\begin{figure}[H] 
		\includegraphics[width=\textwidth]{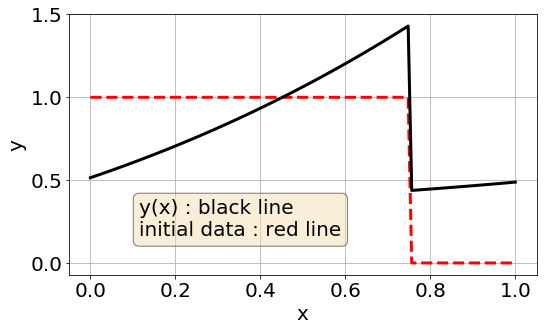}
		\caption{Evolution of a strategy $\bmy$ with $\mca(\bmy)<\frac{1}{2}$.}
		\label{fig:stepfunc}
	\end{figure}
\end{minipage}%
\hfill
\begin{minipage}[b]{0.47\textwidth}
	\begin{figure}[H]
		\centering
		\includegraphics[width=\textwidth]{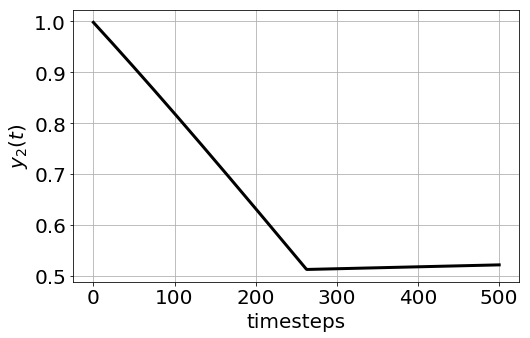}
		\caption{The second component of $\bmy$ at time $t$.}
		\label{fig:y2}
	\end{figure}
\end{minipage}\\[0.8em]

\noindent
\begin{minipage}[b]{0.47\textwidth}
	\begin{figure}[H] 
		\includegraphics[width=\textwidth]{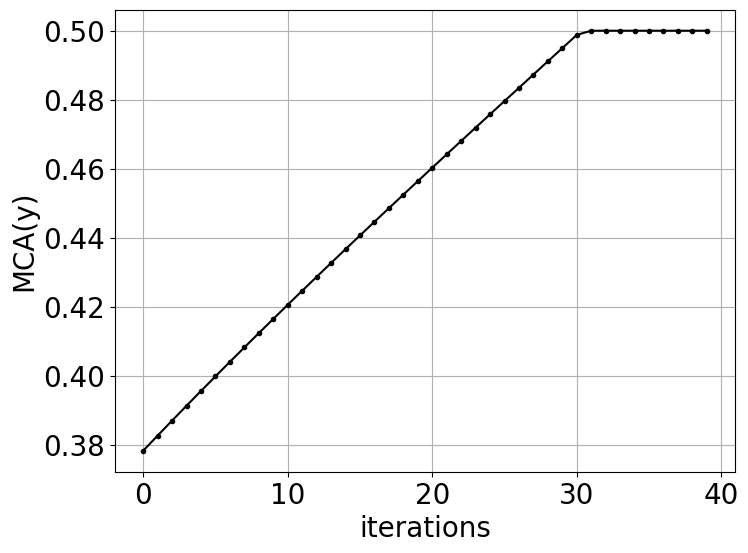}
		\caption{Evolution with integration step size $\varepsilon=0.02$ of a strategy $\bmy$ with $\mca(\bmy)<\frac{1}{2}$.}
		\label{fig:mca-evolution-1}
	\end{figure}
\end{minipage}%
\hfill
\begin{minipage}[b]{0.47\textwidth}
	\begin{figure}[H]
		\centering
		\includegraphics[width=\textwidth]{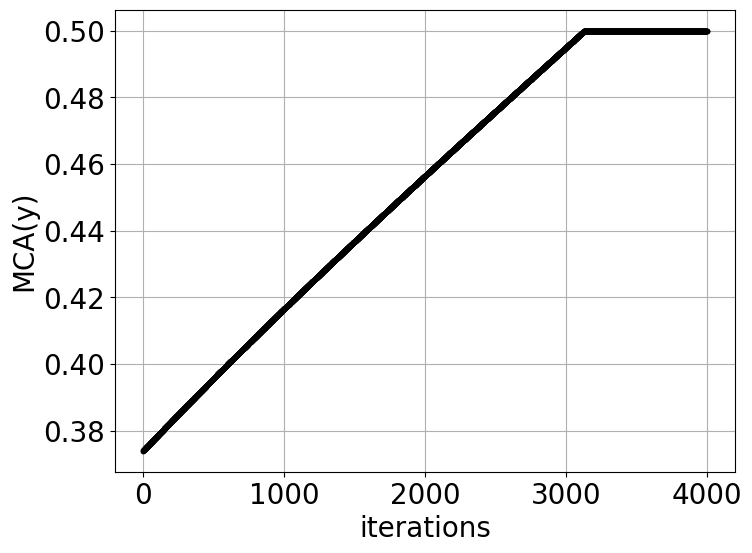}
		\caption{Evolution with integration step size $\varepsilon=0.0002$ of a strategy $\bmy$ with $\mca(\bmy)<\frac{1}{2}$.}
		\label{fig:mca-evolution-2}
	\end{figure}
\end{minipage}\\[0.8em]



%
%

\subsection{New perspectives on adaptive dynamics and possible experimental verifications}\label{sec:perspective}
The distribution of competitive ability within a species can be the result of either standing genetic variation or arise from a monoclonal population where the phenotype for each individual (i.e.\ the competitive ability) has a stochastic element and follows a certain distribution. The latter interpretation is the standard one within adaptive dynamics, i.e.\ a monoclonal population with a certain function-valued trait (the distribution of competitive abilities) is invaded (or not) by a mutant with slightly different distribution \cite{brannstrom2013hitchhiker}. However, the former interpretation is also reasonable, but here change in the resident population is possible not only via invasion, but also acquired mutations that alter the distribution of competitive abilities. How this interpretation should be treated from a mathematical perspective is not clear, since the standard adaptive dynamics framework cannot capture intraspecies genetic variability. 



To investigate how well the team games and the predictions of adaptive dynamics fit with empirical observations, one could conduct experiments involving competition between strains or species. In order to test the predictions of the team game, we suggest that the following conditions should be met:
\begin{enumerate}
	\item Two or more species (or strains) that compete for the same resources should be studied simultaneously. They need to be asexual and they should ideally reside in a relatively homogeneous environment such that spatial or temporal separation is  unlikely. Likewise, each member of a species should be able to compete with any member of any other species. There should not be any ``protected groups'' in the ecological system. 
	\item The competitive ability of the species needs to be observable. Moreover, it needs to be quantitative; the competitive ability of an individual should be represented by a number. It is possible that the competitive ability is a compound ability (consisting of several abilities) as long as the individuals can be ordered from low competitive ability to high competitive ability. This ensures that the arguments in Section~\ref{sec:got} can be carried out, which is necessary for the theoretical setting. 
	\item Every species' mean competitive ability should be bounded by the same value. This requirement is due to the mathematical constraint as explained in Section~\ref{sec:constraints}.
	\item The experiment has to run sufficiently long time that evolution can be observed. This allows for observations of the dynamics of evolution. The initial value problem~(\ref{eq:problem}) can be tested over time if the competitive ability of the species can be observed as it (and \emph{if} it) changes over time. 
	\item Mutations that affect the considered trait have to be rare enough so that genetic variation does not arise during the experiment, which could confound the interpretation of the results.
\end{enumerate}

The results of such a study could reveal whether or not the theory presented in this paper can explain the evolutionary dynamics of certain species. In particular our theory could be a helpful tool in  biogeography, where it is recommended that manipulative experiments and temporal data sets are to be combined with theoretical models in order to explore the diversity and composition of species \cite{hanson2012}.

\subsection{The paradox of the plankton}
The vast amount of microbial species is seemingly a paradox from the theoretical point of view \cite{hutchinson1961paradox}. Models in competition theory have described that a number of species is limited by the number of key resources. This should raise the concern that such models fail to describe the ecology of microbes. In this work, we have aimed at presenting a model that allows a vast number of species to simultaneously co-exist and compete for survival. As it turns out, the adaptive dynamics applied to the game of teams is constantly changing the composition of a species for nearly all species.  The only exception to this are the species characterized by the equilibrium strategies of the game.  
These equilibrium strategies are the only stationary points of the adaptive dynamics.  However, there is no stability in the dynamics in the sense that any perturbation of a stationary point will unsettle the dynamics.  These results align well with the idea that evolution does not stabilize and moreover, it does not put a restriction on the number of species.  Mechanisms such as genetic drift will make sure that any species accrues DNA changes. Even ``living fossils'' such as coelacanths are never static \cite{woolston2013}.  This also fits with ``biology's first law'' the tendency for diversity and complexity to increase in evolving systems \cite{McShea2010}.

\subsection{Outlook}
We have applied adaptive dynamics to the game of teams and accurately described the adaptive game both as a function-valued game and as a discrete, vector-valued game.  In both cases, we showed the existence of solutions and identified their essential characteristics. We also explored differences and similarities between the adaptive dynamics for the discrete game and the function-valued game.
Carefully chosen examples were used in order to highlight important aspects of the dynamics.  One major result of this work is that the stationary points of the dynamics, in both the function-valued and discrete vector-valued case, are precisely the equilibrium strategies of the associated team game.  Although these are stationary points, they are not stable, and branching may occur when one takes perturbations of the stationary solution.  Further studies would be needed to investigate how well these team games and the current adaptive dynamics agree with populations of organisms that are found in nature. 
Our results provide a rich basis of characteristics that can be tested in experiments.

\subsection{Python implementation}
We provide a Python module that can be used for solving the adaptive dynamics problem as described in Section~\ref{sec:computations}.
It is available here:
\begin{center}
 \url{https://github.com/carljoar/to-appear}
\end{center}



\end{document}